\documentclass[a4paper,11pt]{article}

\usepackage{mathtools,amsthm,amssymb,mathrsfs,amsfonts,newtxmath,physics}
\usepackage{tikz,graphicx,multirow,float,enumitem}
\usepackage{fancyhdr,titlesec,hyperref,geometry}
\usepackage[normalem]{ulem}

\usetikzlibrary{matrix,positioning,arrows.meta}

\usepackage{array}
\newcolumntype{D}{>{$\displaystyle}c<{$}}

\AtBeginDocument{%
    \setlength{\abovedisplayskip}{9pt plus 2pt minus 2pt}%
    \setlength{\belowdisplayskip}{9pt plus 2pt minus 2pt}%
    \setlength{\abovedisplayshortskip}{3pt plus 3pt}%
    \setlength{\belowdisplayshortskip}{7pt plus 2pt minus 2pt}%
}

\titleformat{name=\subsection}[runin]
{\large\normalfont\bfseries}{\thesubsection.}{0.5em}{}[]
\title{\huge\textbf{\textsf{The Separated Orlicz Bump Theorem in $L^2$}}}
\author{Zhou Yuekai}
\date{October 23, 2026}

\newtheoremstyle{BTFFS}
    {8pt}{8pt}{\normalfont}{}{\bfseries}{}{.5em}{}
\theoremstyle{BTFFS}
\newtheorem{theorem}{Theorem}[section]
\newtheorem*{theorem*}{Theorem}
\newtheorem{lemma}[theorem]{Lemma}
\newtheorem{corollary}[theorem]{Corollary}
\newtheorem{proposition}[theorem]{Proposition}

\newtheorem{definition}[theorem]{Definition}
\newtheorem{example}[theorem]{Example}
\newtheorem{notes}[theorem]{Note}
\newtheorem{remark}[theorem]{Remark}

\newtheorem{condition}[theorem]{Condition}
\newtheorem{assumption}[theorem]{Assumption}

\setlist[itemize]{itemsep=0pt,parsep=0pt}
\setlist[enumerate]{itemsep=0pt,parsep=0pt}

\newenvironment{dingli*}[1][]
{\begin{theorem*}[#1]\setlength{\leftskip}{30pt}\setlength{\rightskip}{30pt}}
{\end{theorem*}}

\newenvironment{zhengming*}[1][]
{\begin{proof}[#1]\setlength{\leftskip}{30pt}\setlength{\rightskip}{30pt}}
{\end{proof}}

\newcommand{\RR}{\mathbb{R}}

\newcommand{\NN}{\mathbb{N}}
\newcommand{\ZZ}{\mathbb{Z}}

\newcommand{\car}{\mathrm{Car}}

\DeclarePairedDelimiter{\xkhdelim}{(}{)}
\DeclarePairedDelimiter{\lkhdelim}{\langle}{\rangle}

\DeclarePairedDelimiter{\dakhdelim}{\lbrace}{\rbrace}

\newcommand{\xkh}[1]{\xkhdelim*{#1}}
\newcommand{\lkh}[1]{\lkhdelim*{#1}}

\newcommand{\dakh}[1]{\dakhdelim*{#1}}
\newcommand{\qw}[2]{\mathbb{E}\left( #1 \mid #2 \right)}

\newcommand{\EE}{\mathbb{E}}
\newcommand{\PP}{\mathbb{P}}

\newcommand{\cD}{\mathcal{D}}

\newcommand{\cF}{\mathcal{F}}
\newcommand{\cG}{\mathcal{G}}

\newcommand{\cP}{\mathcal{P}}
\newcommand{\cQ}{\mathcal{Q}}

\newcommand{\dmu}{\dd{\mu}}
\newcommand{\dnu}{\dd{\nu}}

\newcommand{\dt}{\dd{t}}
\newcommand{\dw}{\dd{\omega}}

\DeclareMathOperator*{\esssup}{ess\,sup}
\DeclareMathOperator*{\essinf}{ess\,inf}

\newenvironment{thm}[1][]
{\begin{theorem}[#1]\setlength{\leftskip}{30pt}\setlength{\rightskip}{30pt}}
{\end{theorem}}
\newenvironment{lem}
{\begin{lemma}\setlength{\leftskip}{30pt}\setlength{\rightskip}{30pt}}
{\end{lemma}}
\newenvironment{prop}
{\begin{proposition}\setlength{\leftskip}{30pt}\setlength{\rightskip}{30pt}}
{\end{proposition}}
\newenvironment{defn}
{\begin{definition}\setlength{\leftskip}{30pt}\setlength{\rightskip}{30pt}}
{\end{definition}}
\newenvironment{cor}
{\begin{corollary}\setlength{\leftskip}{30pt}\setlength{\rightskip}{30pt}}
{\end{corollary}}
\newenvironment{rem}
{\begin{remark}\setlength{\leftskip}{30pt}\setlength{\rightskip}{30pt}}
{\end{remark}}

\titleformat{\section}{\Large\normalfont\bfseries}{\thesection.}{0.5em}{}

\newcommand{\proofpart}{\par\bigskip\noindent\(\blacktriangleright\)\enspace\ignorespaces}

\begin{document}
	
	\maketitle
	
	\begin{abstract}
		We prove a separated entropy bump theorem for two-weight $L^2$ inequalities for positive operators defined by Carleson sequences on a probability space with a general filtration. The bump conditions use conditional averages of martingale maximal functions and involve the first power of the entropy penalty, removing the square appearing in the separated condition of Treil and Volberg. As a consequence, we obtain separated Orlicz bump estimates for a class of doubling Young functions. The final section describes a direct Orlicz argument that removes the doubling assumption and the monotonicity assumption on $\frac{\Phi(t)}{t\log t}$, and discusses extensions to other exponents. These results apply, in particular, to sparse operators. The proof uses principal sequences, a natural global extension of principal sets, and combines a good-$\lambda$ estimate with a weighted Carleson embedding to obtain the required testing inequalities. 
	\end{abstract}
	
\section{Introduction}

Two-weight inequalities for singular integral operators ask for conditions on weights $u$, $v$ that imply
\[
	\int_{\RR^d} \abs{T\xkh{f u}}^2 v \dd{x} \lesssim \int_{\RR^d} \abs{f}^2 u \dd{x}.
\]
A necessary condition for the Hilbert transform is the two-weight $A_2$ condition
\[
	\sup_I \lkh{u}_I \lkh{v}_I < \infty,
\]
where $I$ ranges over intervals. This condition alone is insufficient; see \cite{enma2016}. One therefore seeks sufficient conditions obtained by strengthening the averages in this product. Neugebauer's weight-insertion theorem gave an early approach using power bumps \cite{neugebauer1983}. P\'erez studied Orlicz-norm sufficient conditions for weighted singular integral inequalities \cite{perez1994} and for the two-weight boundedness of the Hardy--Littlewood maximal operator \cite{perez1995}. Cruz-Uribe and P\'erez obtained sharp two-weight weak-type inequalities for singular integral operators \cite{cruzuribe-perez1999}.

For a Young function $\Phi$ and a cube $Q$, define the normalized Luxemburg norm by
\[
	\norm{f}_{L^\Phi(Q)} = \inf\dakh{\lambda > 0 : \frac{1}{\abs{Q}} \int_Q \Phi\xkh{\frac{\abs{f(x)}}{\lambda}} \dd{x} \leq 1}.
\]
In the $L^2$ setting, the two-sided bump condition takes the form
\[
	\sup_Q \norm{u}_{L^{\Phi_1}(Q)} \norm{v}_{L^{\Phi_2}(Q)} < \infty,
\]
where the Young functions satisfy, for some $t_0 > 1$,
\[
	\int_{t_0}^{\infty} \frac{1}{\Phi_i(t)} \dt < \infty, \quad i = 1, 2.
\]
The sufficiency of this condition for Calder\'on--Zygmund operators was proved by Lerner \cite{lerner2013} and, by a Bellman function argument in $L^2$, by Nazarov, Reznikov, Treil and Volberg \cite{nrtv2013}.

The separated bump problem asks whether the two weaker conditions
\[
	\sup_Q \norm{u}_{L^{\Phi_1}(Q)} \lkh{v}_Q < \infty, \quad \sup_Q \lkh{u}_Q \norm{v}_{L^{\Phi_2}(Q)} < \infty
\]
are sufficient for the same estimate. Each condition strengthens only one of the two averages. Cruz-Uribe, Reznikov and Volberg \cite[Conjecture~1]{cruzuribe-reznikov-volberg2014} formulated the separated conjecture and proved it for logarithmic bumps and for a class of iterated logarithmic bumps. Lacey \cite{lacey2016} proved separated estimates with an additional penalty. Entropy versions and further sufficient conditions were obtained in \cite{lacey-spencer2015,rahm-spencer2018,lerner2022}.

Treil and Volberg \cite{enma2016} introduced entropy bumps. For a dyadic cube $I$, let
\[
	u_I^* = \frac{1}{\abs{I}} \int_I \sup_{\substack{J \subseteq I \\ J \text{ dyadic}}} \lkh{u}_J \chi_J \dd{x}.
\]
Their separated theorem for sparse operators assumes
\[
	\sup_I \alpha\xkh{\frac{u_I^*}{\lkh{u}_I}}^2 u_I^* \lkh{v}_I < \infty, \quad \sup_I \alpha\xkh{\frac{v_I^*}{\lkh{v}_I}}^2 v_I^* \lkh{u}_I < \infty,
\]
where $t \mapsto t\alpha(t)$ is increasing and
\[
	\frac{1}{\alpha(1)} + \int_1^{\infty} \frac{1}{t\alpha(t)} \dt < \infty.
\]
They pointed out that removing the square of $\alpha$ would imply the separated bump conjecture for sufficiently regular Young functions, and left open whether the square was necessary \cite[p.~504]{enma2016}.

In this paper, we prove the separated entropy bump theorem with the first power of $\alpha$. We work on a probability space with a general filtration and consider
\[
	T_a f = \sum_{n \in \NN} a_n \EE_n\xkh{f},
\]
where $a$ is a nonnegative Carleson sequence. Positive operators in filtered measure spaces were studied by Tanaka and Terasawa \cite{tanaka-terasawa2013} and Chen, Zhu, Zuo and Jiao \cite{chen2020}. In this setting, put
\[
	u_n = \EE_n\xkh{u}, \quad u_n^M = \EE_n\xkh{M_n u}, \quad M_n u = \sup_{k \geq n} \EE_k\xkh{u},
\]
and use the corresponding notation for $v$. Theorem~\ref{thm:1side} proves the two-weight estimate under
\[
	\sup_{n \in \NN} \esssup \xkh{\alpha\xkh{\frac{u_n^M}{u_n}} u_n^M v_n} < \infty, \quad \sup_{n \in \NN} \esssup \xkh{\alpha\xkh{\frac{v_n^M}{v_n}} v_n^M u_n} < \infty,
\]
with the same monotonicity and integral assumptions on $\alpha$. Theorem~\ref{thm:1side-bump} gives the corresponding separated Orlicz bump theorem for doubling Young functions satisfying the preceding integral condition and such that
\[
	t \mapsto \frac{\Phi_i(t)}{t \log t}, \quad i = 1, 2,
\]
are increasing for sufficiently large $t$. In Section~\ref{sec:further}, we describe a direct Orlicz argument that removes these two additional assumptions and discuss its extension to other exponents. When each $\cF_n$ is generated by atoms, sparse operators are included among the operators $T_a$. Through Lerner's domination theorem \cite{lerner2013}, their estimates also give the corresponding two-weight inequalities for Calder\'on--Zygmund operators; see Remark~\ref{rem:sparse-cz}.

As in the proof of Treil and Volberg \cite{enma2016}, we prove the main theorem through Sawyer-type testing conditions. Such conditions were established by Sawyer for maximal operators \cite{sawyer1982} and fractional and Poisson integrals \cite{sawyer1988}. For positive dyadic operators, the $L^2$ characterization follows from the bilinear embedding theorem of Nazarov, Treil and Volberg \cite{ntv1999}; Lacey, Sawyer and Uriarte-Tuero \cite{lacey-sawyer-uriarte2009} treated $L^p$ to $L^q$ estimates. We prove the sufficiency of conditional Sawyer-type testing conditions for $T_a$. The proof follows the structure of Treil's argument for positive dyadic operators \cite{treil2015positive}, using the principal set construction of \cite{tanaka-terasawa2013,chen2020}.

We introduce principal sequences to express sparse operators directly through the filtration. When each $\cF_n$ is generated by atoms, principal sequences correspond to sparse families. For a principal sequence $b$, the associated operator is
\[
	T_b f = \sum_{n \in \NN} b_n \EE_n\xkh{f}.
\]
We prove a good-$\lambda$ estimate for sums over arbitrary principal sequences and integrate it to obtain a weighted square estimate. We then use operators $T_b$ to dominate general Carleson operators $T_a$ and obtain the corresponding estimates for $T_a$.

Section~\ref{sec:preliminary} introduces the notation and states the main results. Sections~\ref{sec:tools} and~\ref{sec:embedding} contain the auxiliary estimates and the weighted Carleson embedding. We prove the sufficiency of the conditional testing conditions in Section~\ref{sec:sawyer}, develop the principal sequence estimates in Section~\ref{sec:principal}, and complete the proof in Section~\ref{sec:testing}. Section~\ref{sec:further} discusses the direct Orlicz argument and its extension to other exponents.

\section{Preliminaries}
\label{sec:preliminary}

\subsection{Basic definitions}

Consider a probability measure space $(\Omega, \cF, \PP)$ with the filtration $\{\cF_n\}_{n \geq 0}$, where $\cF_n \subset \cF_{n + 1} \subset \cF$ for all $n \in \NN$. No atomicity assumption is imposed in this paper. Denote $\cF_\infty = \bigvee_{n \in \NN} \cF_n$.

We write $\abs{A}$ for $\PP(A)$ and $\dw$ for $\dd{\PP}$ in integrals.

For a measurable set $E$ we define the average
\[
	\lkh{f}_{E} = \frac{1}{\abs{E}} \int_E f \dw.
\]

\begin{defn}
	Let $\cG \subseteq \cF$ be a sub-$\sigma$-algebra such that $\PP|_{\cG}$ is a probability measure. For $f \in L^1(\Omega, \cF, \PP)$, the conditional expectation of $f$ with respect to $\cG$ is the unique function $g \in L^1(\Omega, \cG, \PP)$, up to equality $\PP$-a.e. (almost everywhere), such that
	\[
		\lkh{g}_{A} = \lkh{f}_{A}, \qquad A \in \cG.
	\]
	We write
	\[
		g = \qw{f}{\cG}.
	\]
	For the filtration $(\cF_{n})_{n \geq 0}$, we set
	\[
		f_n = \EE_n\xkh{f} := \qw{f}{\cF_{n}}.
	\]
\end{defn}

The martingale maximal function $M$ and the $\geq n$-maximal function $M_n$ are defined by
\[
	M{f} = \sup_{n \in \NN} \abs{\EE_n\xkh{f}}, \quad M_n{f} = \sup_{k \geq n} \abs{\EE_k\xkh{f}}.
\]

\begin{defn}
	Let $a = (a_n)_{n \in \NN}$ be a sequence of nonnegative functions such that $a_n \in L^1(\cF_n)$ for $n \in \NN$. We say that $a$ is a Carleson sequence if
	\[
		\norm{a}_{\car} := \sup_{n \in \NN} \esssup \EE_{n}\xkh{\sum_{k \geq n}{a_k}} < \infty.
	\]
	We call $\norm{a}_{\car}$ the Carleson norm of $a$.
\end{defn}

A family $\cQ \subset \cF$ is called a sparse family if for any $I \in \cQ$
\[
	\abs{\bigcup_{J \in \cQ, J \subsetneq I} J} \leq \frac{1}{2}\abs{I}.
\]
The constant $\frac{1}{2}$ can be replaced by any constant $\eta \in (0, 1)$, and the corresponding family is called an $\eta$-sparse family. In this paper, we only consider the case $\eta = \frac{1}{2}$.

\begin{defn}
	Let $\cQ \subset \cF$ be a sparse family. A sparse operator $T = T_\cQ$ is defined by
	\[
		T_\cQ{f} = \sum_{I \in \cQ} \lkh{f}_{I} \chi_{I}.
	\]
\end{defn}

To work with general filtrations, we consider operators defined directly from Carleson sequences. 

\begin{defn}
	Let $a = (a_n)_{n \in \NN}$ be a Carleson sequence. We define the operator $T_a$ by
	\begin{equation}
		\label{def:CARoper}
		T_a{f} = \sum_{n \in \NN} a_n \EE_n\xkh{f}.
	\end{equation}
\end{defn}

In the atomic setting, sparse operators are special cases of $T_a$. We use $T_a$ in the main theorems and their proofs.

\subsection{Main results}

For a weight $u$, we denote
\[
	u_{n} := \qw{u}{\cF_n}, \quad u_{n}^{M} := \qw{\abs{M_n u}}{\cF_n},
\]
and use analogous notation for a weight $v$. $M_{\varphi}$ is the multiplication operator, $M_{\varphi} f = \varphi f$. 

For a Young function $\Phi$, we denote the conditional Luxemburg norm by
\[
	\norm{f}^{\cF_n}_{L^\Phi} := \essinf \dakh{\lambda > 0 : \EE_n\xkh{\Phi\xkh{\frac{\abs{f}}{\lambda}}} \leq 1},
\]
where $\lambda$ ranges over positive $\cF_n$-measurable functions.

\begin{thm}
	\label{thm:1side-bump}
	Let $a$ be a Carleson sequence and let $T = T_a$ be the corresponding operator. Let $\Phi_1$, $\Phi_2$ be doubling Young functions such that, for some $t_0 > 1$,
	\[
		\int_{t_0}^{\infty} \frac{1}{\Phi_i(t)} \dt < \infty, \quad i = 1, 2,
	\]
	and the functions $t \mapsto \frac{\Phi_i(t)}{t \log t}$ are increasing on $[t_0, \infty)$. Suppose that the weights $u$, $v$ satisfy the one-sided bump conditions
	\[
		\sup_{n \in \NN} \esssup \xkh{\norm{u}^{\cF_n}_{L^{\Phi_1}} v_n} < \infty, \quad \sup_{n \in \NN} \esssup \xkh{\norm{v}^{\cF_n}_{L^{\Phi_2}} u_n} < \infty.
	\]
	Then the operator $M_v^{\frac{1}{2}} T M_u^{\frac{1}{2}}$ is bounded on $L^2$, i.e.
	\[
		\int_{\Omega} \abs{T\xkh{f u}}^2 v \dw \lesssim \int_{\Omega} \abs{f}^2 u \dw, \qquad f \in L^2(u).
	\]
	The implicit constant depends only on $\Phi_1$, $\Phi_2$, the two bump bounds and $\norm{a}_{\car}$.
\end{thm}

This theorem follows from the separated entropy bump theorem below. 

\begin{thm}
	\label{thm:1side}
	Let $a$ be a Carleson sequence and let $T = T_a$ be the operator defined by \eqref{def:CARoper}. Let the weights $u$, $v$ satisfy the separated bump conditions
	\[
		\sup_{n \in \NN} \esssup \xkh{\alpha\xkh{\frac{u_{n}^{M}}{u_n}} u^{M}_{n} v_n} = A_u < \infty, \quad \sup_{n \in \NN} \esssup \xkh{\alpha\xkh{\frac{v_{n}^{M}}{v_n}} v^{M}_{n} u_{n}} = A_v < \infty,
	\]
	where $\alpha \colon [1, \infty) \to \RR_{+}$ is such that the function $t \mapsto t\alpha(t)$ is increasing and
	\begin{equation}
		\label{def:Ca}
		C_{\alpha} := \frac{1}{\alpha(1)} + \int_1^{\infty} \frac{1}{t\alpha(t)} \dt < \infty.
	\end{equation}
	Then the operator $M_v^{\frac{1}{2}} T M_u^{\frac{1}{2}}$ is bounded on $L^2$, i.e.
	\[
		\int_{\Omega} \abs{T\xkh{f u}}^2 v \dw \lesssim \int_{\Omega} \abs{f}^2 u \dw, \qquad f \in L^2(u).
	\]
\end{thm}

\section{General Tools}
\label{sec:tools}

In this paper, we denote
\[
\rho_n :=
\begin{cases}
	\displaystyle \frac{u_n^M}{u_n} , & \text{where } u_n > 0; \\
	1 , & \text{where } u_n = 0. 
\end{cases}
\]

\begin{rem}
	The real-variable facts about Young functions and the function $\alpha$ used below are taken from Section~3 of \cite{enma2016}. Likewise, the Bellman function used in the Carleson embedding theorem is provided by Lemma~4.3 of \cite{enma2016}. We state the facts needed here and do not reproduce their constructions.
\end{rem}

\subsection{Young functions and entropy bumps}

For a measurable function $f$, write
\[
	N_f(t) = \abs{\dakh{\abs{f} > t}}, \quad N_f^{\cF_n}(t) = \EE_n\xkh{\chi_{\dakh{\abs{f} > t}}}, \quad t \geq 0.
\]
Following \cite[Section~3.2]{enma2016}, we equip $L\log L$ with the Lorentz norm determined by $\psi_0(s) = s\log(\frac{e}{s})$ for $0 < s \leq 1$ and $\psi_0(0) = 0$. Thus,
\[
	\norm{f}_{L\log L} = \int_0^{\infty} \psi_0\xkh{N_f(t)} \dt, \quad \norm{f}^{\cF_n}_{L\log L} = \int_0^{\infty} \psi_0\xkh{N_f^{\cF_n}(t)} \dt.
\]

The following result is Lemma~3.3 of \cite{enma2016}.

\begin{lem}
	\label{lem:YoungAlpha}
	Let $\Phi$ be a doubling Young function such that, for some $t_0 > 1$,
	\[
		\int_{t_0}^{\infty} \frac{1}{\Phi(t)} \dt < \infty,
	\]
	and $t \mapsto \frac{\Phi(t)}{t\log t}$ is increasing on $[t_0, \infty)$. Then there exists $\alpha \colon [1, \infty) \to \RR_+$ such that $t \mapsto t\alpha(t)$ is increasing,
	\[
		C_{\alpha} := \frac{1}{\alpha(1)} + \int_1^{\infty} \frac{1}{t\alpha(t)} \dt < \infty,
	\]
	and
	\[
		\alpha\xkh{\frac{\norm{f}_{L\log L}}{\norm{f}_{L^1}}} \norm{f}_{L\log L} \leq \norm{f}_{L^\Phi}, \qquad 0 \neq f \in L^\Phi.
	\]
\end{lem}

For a measurable set $E$, the conditional weak $(1,1)$ inequality gives
\begin{align*}
	\EE_n\xkh{M_n\chi_E} & = \int_0^1 \EE_n\xkh{\chi_{\dakh{M_n\chi_E > s}}} \dd{s} \\
	& \leq \int_0^1 \min\dakh{1, \frac{\EE_n\xkh{\chi_E}}{s}} \dd{s} \\
	& = \psi_0\xkh{\EE_n\xkh{\chi_E}}.
\end{align*}
Since $u = \int_0^{\infty} \chi_{\dakh{u > t}} \dt$, we have
\begin{align*}
	u_n^M & = \EE_n\xkh{M_n\xkh{\int_0^{\infty} \chi_{\dakh{u > t}} \dt}} \\
	& \leq \int_0^{\infty} \EE_n\xkh{M_n\chi_{\dakh{u > t}}} \dt \\
	& \leq \int_0^{\infty} \psi_0\xkh{N_u^{\cF_n}(t)} \dt = \norm{u}^{\cF_n}_{L\log L}.
\end{align*}
Applying Lemma~\ref{lem:YoungAlpha} to the conditional distribution $N_u^{\cF_n}(\omega, \cdot)$ and using the monotonicity of $t\alpha(t)$, we obtain
\begin{align*}
	\alpha\xkh{\frac{u_n^M}{u_n}} u_n^M & \leq \alpha\xkh{\frac{\norm{u}^{\cF_n}_{L\log L}}{u_n}} \norm{u}^{\cF_n}_{L\log L} \\
	& \leq \norm{u}^{\cF_n}_{L^\Phi}.
\end{align*}

For $\Phi_1$, $\Phi_2$, take $\alpha_1$, $\alpha_2$ as above and set $\alpha = \min\dakh{\alpha_1, \alpha_2}$. Then $t\alpha(t)$ is increasing, and
\begin{gather*}
	C_{\alpha} = \frac{1}{\alpha(1)} + \int_1^{\infty} \frac{1}{t\alpha(t)} \dt \leq C_{\alpha_1} + C_{\alpha_2}, \\
	\alpha\xkh{\frac{u_n^M}{u_n}} u_n^M v_n \leq \alpha_1\xkh{\frac{u_n^M}{u_n}} u_n^M v_n \leq \norm{u}^{\cF_n}_{L^{\Phi_1}} v_n, \\
	\alpha\xkh{\frac{v_n^M}{v_n}} v_n^M u_n \leq \alpha_2\xkh{\frac{v_n^M}{v_n}} v_n^M u_n \leq \norm{v}^{\cF_n}_{L^{\Phi_2}} u_n.
\end{gather*}
Thus Theorem~\ref{thm:1side} implies Theorem~\ref{thm:1side-bump}.

\subsection{A useful function}

In the proof of the separated entropy bump theorem, we shall use a decreasing integrable function $\psi$ satisfying
\[
	\psi(t) \geq \psi(s)\frac{s^2}{t^2}, \qquad t \geq s \geq 1.
\]

To construct that function, we define
\begin{equation}
	\label{def:psi}
	\psi(t) = \frac{\displaystyle 1 + 2\alpha(1)\int_1^t \frac{1}{\alpha(s)} \dd{s}}{t^2}, \qquad t \geq 1.
\end{equation}
Since $t\alpha(t)$ is increasing,
\begin{align*}
	t^2\psi(t) & = 1 + 2\alpha(1)\int_1^t \frac{s}{s\alpha(s)} \dd{s} \\
	& \geq 1 + \frac{2\alpha(1)}{t\alpha(t)}\int_1^t s \dd{s} \\
	& = 1 + \frac{\alpha(1)(t^2 - 1)}{t\alpha(t)} \\
	& \geq \frac{\alpha(1)t^2}{t\alpha(t)}.
\end{align*}
The function $\psi$ is locally absolutely continuous, and
\begin{align*}
	\psi(1) = 1, & \qquad \psi(t) \geq \frac{\alpha(1)}{t\alpha(t)}, \\
	\psi'(t) = \frac{2}{t}\xkh{\frac{\alpha(1)}{t\alpha(t)} - \psi(t)}, & \qquad 0 \leq -\psi'(t) \leq \frac{2\psi(t)}{t}.
\end{align*}
For $t \geq s \geq 1$,
\begin{align*}
	t^2\psi(t) & = 1 + 2\alpha(1)\int_1^t \frac{1}{\alpha(r)} \dd{r} \\
	& \geq 1 + 2\alpha(1)\int_1^s \frac{1}{\alpha(r)} \dd{r} = s^2\psi(s).
\end{align*}
Also,
\begin{align*}
	C_{\psi} := 1 + \int_1^{\infty} \psi(t) \dt & = 2 + 2\alpha(1)\int_1^{\infty} \frac{1}{\alpha(s)} \int_s^{\infty} \frac{1}{t^2} \dt \dd{s} \\
	& = 2 + 2\alpha(1)\int_1^{\infty} \frac{1}{s\alpha(s)} \dd{s} = 2\alpha(1)C_{\alpha}.
\end{align*}

\begin{rem}
	The related scaling condition $t^2 \psi(t) \geq s^2 \psi(s)$ with $t \geq s \geq 1$ is used in Lemma~4.1 of \cite{nrtv2013}. There, $\Psi$ is decreasing, $s\Psi(s)$ is increasing, and
	\[
		\int_0^1 \frac{1}{s\Psi(s)} \dd{s} < \infty.
	\]
	Setting
	\[
		\widetilde{\psi}(t) = \frac{1}{t\Psi\xkh{\frac{1}{t}}}, \quad t \geq 1,
	\]
	we obtain a decreasing integrable function such that $t^2\widetilde{\psi}(t)$ is increasing and $t\widetilde{\psi}(t)$ is decreasing. For their auxiliary function
	\[
		T(A, N) = N \int_0^{\frac{N}{A}} \frac{1}{s\Psi(s)} \dd{s}, \quad 1 \leq A \leq 2, \quad 0 < N \leq 1,
	\]
	the scaling condition gives
	\[
		-\frac{\partial T}{\partial A}(A, N) = \widetilde{\psi}\xkh{\frac{A}{N}} \geq \frac{1}{A^2} \widetilde{\psi}\xkh{\frac{1}{N}} \geq \frac{1}{4} \widetilde{\psi}\xkh{\frac{1}{N}}.
	\]
	They use this estimate to prove the differential inequality for their Bellman function in the Carleson embedding theorem.
\end{rem}

\begin{rem}
	The function $\psi$ can also be obtained from the Bellman function construction in \cite[Section~4.2]{enma2016}. Set $\varphi(t) = \frac{1}{\alpha(1)}$ for $0 < t < 1$ and $\varphi(t) = \frac{1}{t\alpha(t)}$ for $t \geq 1$. In the proof of Lemma~4.3 there, the Bellman function has the form $B(x, t) = x^2 m(t)$, where a nonnegative measure $\mu$ gives
	\begin{gather*}
		\varphi(t) = \int_0^{\infty} \frac{\chi_{[0, r]}(t)}{r} \dd{\mu(r)}, \\
		m(t) = 4 \int_0^{\infty} \frac{r}{r + t} \dd{\mu(r)}, \qquad -m'(t) = 4 \int_0^{\infty} \frac{r}{(r + t)^2} \dd{\mu(r)}.
	\end{gather*}
	Using the same representation, we can rewrite \eqref{def:psi} as
	\begin{align*}
		\psi(t) & = \frac{2\alpha(1)}{t^2} \int_0^t s\varphi(s) \dd{s} \\
		& = \frac{2\alpha(1)}{t^2} \int_0^{\infty} \frac{1}{r} \int_0^{\min\dakh{t, r}} s \dd{s} \dd{\mu(r)} \\
		& = \alpha(1) \int_0^{\infty} \frac{r}{\max\dakh{r, t}^2} \dd{\mu(r)}, \qquad t \geq 1.
	\end{align*}
	Since $\max\dakh{r, t} \leq r + t \leq 2\max\dakh{r, t}$, we obtain
	\[
		\frac{\alpha(1)}{4}\xkh{-m'(t)} \leq \psi(t) \leq \alpha(1)\xkh{-m'(t)}.
	\]
	The monotonicity of $\psi(t)$ and of $t^2\psi(t)$ also follows directly from this integral representation.
\end{rem}

We define
\[
	G(r) = \int_1^r \psi(t) \dt, \quad G(\infty) = C_{\psi} - 1,
\]
and
\[
	d(s, r) = G(r) - G(s) - \psi(r)(r - s), \qquad 1 \leq s, r < \infty.
\]
The function $G$ is increasing and concave, so $d(s, r) \geq 0$. 

\begin{lem}
	\label{lem:d}
	If $s \geq 1$, $r \in [1, \infty]$ and $\psi(s) > 2\psi(r)$, then
	\[
		d(s, r) \geq \frac{s\psi(s)}{48}.
	\]
\end{lem}

\begin{proof}
	By the monotonicity of $\psi$ and the comparison above, we have
	\begin{align*}
		r > s \geq 1, & \quad \psi(r) \geq \psi(s) \frac{s^2}{r^2} > 2 \psi(r) \frac{s^2}{r^2}, \quad r > \sqrt{2}s > \frac{4s}{3}; \\
		t \in \xkh{s, \frac{4}{3} s}, & \quad \psi(t) \geq \psi\xkh{\frac{4}{3} s} \geq \frac{9}{16} \psi(s).
	\end{align*}
	Thus,
	\begin{align*}
		d(s, r) & = \int_s^r \xkh{\psi(t) - \psi(r)} \dt \\
		& \geq \int_s^{\frac{4}{3}s} \xkh{\psi(t) - \psi(r)} \dt \\
		& \geq \int_s^{\frac{4}{3}s} \xkh{\frac{9}{16}\psi(s) - \frac{1}{2}\psi(s)} \dt \\
		& = \int_s^{\frac{4}{3}s} \frac{\psi(s)}{16} \dt = \frac{s\psi(s)}{48}. \qedhere
	\end{align*}
\end{proof}

\begin{rem}
	Theorem~2.1.5 of \cite{nesterov2018} gives a stronger form of Lemma~\ref{lem:d}. For $s \leq x \leq y \leq r$,
	\[
		0 \leq \psi(x) - \psi(y) \leq \psi(x)\xkh{1 - \frac{x^2}{y^2}} \leq \frac{2\psi(s)}{s}(y - x).
	\]
	Thus $-G$ has a $\frac{2\psi(s)}{s}$-Lipschitz derivative on $[s, r]$. Applying the theorem to $-G$, extended linearly outside $[s, r]$, we obtain
	\[
		d(s, r) \geq \frac{s\xkh{\psi(s) - \psi(r)}^2}{4\psi(s)} > \frac{s\psi(s)}{16}, \qquad \psi(s) > 2\psi(r).
	\]
\end{rem}

\subsection{A simple application}

Let $\dmu = u\dw$ and $\EE_{\mu}\xkh{X \mid \cF_n} = \frac{\EE_n\xkh{u X}}{u_n}$ for a nonnegative function $X$. Since $M_{n + 1}u \leq M_nu$,
\begin{align*}
	\EE_{\mu}\xkh{\rho_{n + 1} \mid \cF_n} & = \frac{\EE_n\xkh{u_{n + 1} \rho_{n + 1}}}{u_n} \\
	& = \frac{\EE_n\xkh{M_{n + 1}u}}{u_n} \leq \frac{\EE_n\xkh{M_nu}}{u_n} = \rho_n.
\end{align*}
By the concavity and monotonicity of $G$,
\[
	\EE_{\mu}\xkh{G(\rho_{n + 1}) \mid \cF_n} \leq G\xkh{\EE_{\mu}\xkh{\rho_{n + 1} \mid \cF_n}} \leq G(\rho_n).
\]

We define
\[
	e_n = G(\rho_n) - \EE_{\mu}\xkh{G(\rho_{n + 1}) \mid \cF_n} \geq 0. 
\]
Multiplying the definition of $e_n$ by $u_n$, we have
\begin{align*}
	u_n e_n & = u_n G(\rho_n) - u_n \EE_{\mu}\xkh{G(\rho_{n + 1}) \mid \cF_n} \\
	& = u_n G(\rho_n) - \EE_n\xkh{u_{n + 1} G(\rho_{n + 1})} \leq (C_{\psi} - 1)u_n,
\end{align*}
and summing gives
\begin{equation}
	\label{eq:defect-total}
	\begin{aligned}
		\EE_{\ell}\xkh{\sum_{n \geq \ell} u_n e_n} & = \sum_{n \geq \ell} \xkh{\EE_{\ell}\xkh{u_n G(\rho_n)} - \EE_{\ell}\xkh{u_{n + 1} G(\rho_{n + 1})}} \\
		& = u_{\ell} G(\rho_{\ell}) - \EE_{\ell}\xkh{\lim_{n \to \infty} u_n G(\rho_n)} \\
		& \leq (C_{\psi} - 1)u_{\ell}.
	\end{aligned}
\end{equation}

Next, we consider another estimate that follows directly from the preceding discussion.

\begin{cor}
	\label{cor:stopped-defect}
	Let $\sigma \leq \tau$ be stopping times. Then
	\begin{align}
		\EE_{\sigma}\xkh{\chi_{\dakh{\tau < \infty}} u_{\tau} d\xkh{\rho_{\tau}, \rho_{\sigma}}} & \leq \EE_{\sigma}\xkh{\sum_{\sigma \leq n < \tau} u_n e_n}, \label{eq:stopped-defect} \\
		\EE_{\sigma}\xkh{\chi_{\dakh{\tau < \infty, \psi(\rho_{\tau}) > 2 \psi(\rho_{\sigma})}} u_{\tau} \rho_{\tau} \psi\xkh{\rho_{\tau}}} & \leq 48 \cdot \EE_{\sigma}\xkh{\sum_{\sigma \leq n < \tau} u_n e_n}. \label{eq:stopped-jump}
	\end{align}
\end{cor}

To prove this corollary, we only need to note that
\begin{align*}
	\EE_{\mu}\xkh{\chi_{\dakh{\tau < \infty, \psi(\rho_{\tau}) > 2\psi(\rho_{\sigma})}} \frac{\rho_{\tau} \psi(\rho_{\tau})}{48} \mid \cF_{\sigma}} & \leq \EE_{\mu}\xkh{\chi_{\dakh{\tau < \infty}} d\xkh{\rho_{\tau}, \rho_{\sigma}} \mid \cF_{\sigma}} \\
	& \leq G(\rho_{\sigma}) - \EE_{\mu}\xkh{G(\rho_{\tau}) \mid \cF_{\sigma}} \\
	& = \EE_{\mu}\xkh{\sum_{\sigma \leq n < \tau} e_n \mid \cF_{\sigma}},
\end{align*}
and multiply by $u_{\sigma}$.

\section{Carleson Embedding}
\label{sec:embedding}

We prove the embedding theorem for a nonnegative decreasing function $\varphi$, so that it can be applied both to $\varphi(t) = \frac{1}{t\alpha(t)}$ and to the function $\psi$ above.

\begin{thm}
	\label{spEMB}
	Let $a = (a_k)_{k \in \NN}$ be a Carleson sequence and let $\varphi \colon [1, \infty) \to [0, \infty)$ be non-increasing with
	\[
		C_{\varphi} := \varphi(1) + \int_1^{\infty} \varphi(t) \dt < \infty.
	\]
	Then for any $f \in L^2(u)$ and $\ell \in \NN$,
	\begin{equation}
		\label{eq:emb}
		\EE_{\ell}\xkh{\sum_{n \geq \ell} a_n \frac{\abs{\EE_n\xkh{f u}}^2}{u_n} \varphi(\rho_n)} \leq 4 C_{\varphi} \norm{a}_{\car} \EE_{\ell}\xkh{\abs{f}^2 u}.
	\end{equation}
\end{thm}

\begin{lem}
	\label{lem:carPleqMAX}
	Let $a = (a_k)_{k \in \NN}$ be a Carleson sequence and let $(u_k)_{k \in \NN}$ be a nonnegative adapted sequence. For any $n \in \NN$, we have
	\[
		\qw{\sum_{k \geq n} a_k u_k}{\cF_n} \leq \norm{a}_{\car} \qw{\sup_{k \geq n} u_k}{\cF_n}.
	\]
	In particular,
	\[
		\qw{\sum_{k \geq n} a_k u_k}{\cF_n} \leq \norm{a}_{\car} u_n^M.
	\]
\end{lem}

\begin{proof}[Proof of Lemma~\ref{lem:carPleqMAX}]
	Recall that the Carleson condition gives
	\[
		\qw{\sum_{k \geq n} a_k}{\cF_n} \leq \norm{a}_{\car}.
	\]
	Consider the stopping times $\tau_t = \inf\dakh{k \geq n : u_k > t}$. Since $\dakh{\tau_t = j} \in \cF_j$ for $j \geq n$, we have
	\begin{align*}
		\qw{\sum_{k \geq n} a_k \chi_{\dakh{u_k > t}}}{\cF_n} & \leq \sum_{j \geq n} \qw{\chi_{\dakh{\tau_t = j}} \sum_{k \geq j} a_k}{\cF_n} \\
		& = \sum_{j \geq n} \qw{\chi_{\dakh{\tau_t = j}} \qw{\sum_{k \geq j} a_k}{\cF_j}}{\cF_n} \\
		& \leq \norm{a}_{\car} \qw{\chi_{\dakh{\tau_t < \infty}}}{\cF_n}.
	\end{align*}
	Since $\dakh{\tau_t < \infty} = \dakh{\sup_{k \geq n} u_k > t}$, the conditional layer cake formula gives
	\begin{align*}
		\qw{\sum_{k \geq n} a_k u_k}{\cF_n} & = \int_0^{\infty} \qw{\sum_{k \geq n} a_k \chi_{\dakh{u_k > t}}}{\cF_n} \dt \\
		& \leq \norm{a}_{\car} \int_0^{\infty} \qw{\chi_{\dakh{\tau_t < \infty}}}{\cF_n} \dt \\
		& = \norm{a}_{\car} \qw{\sup_{k \geq n} u_k}{\cF_n}. \qedhere
	\end{align*}
\end{proof}

\begin{proof}[Proof of Theorem~\ref{spEMB}]
	Extend $\varphi$ to $[0, 1)$ by $\varphi(t) = \varphi(1)$. By Lemma~4.3 and Section~4.2 of \cite{enma2016}, there is a Bellman function $B(x, y)$ for $x \in \RR$, $y \geq 0$ such that
	\begin{gather*}
		\text{$B$ is convex;} \\
		-\frac{\partial B}{\partial y}(x, y) \geq x^2\varphi(y); \\
		0 \leq B(x, y) \leq 4 C_{\varphi} x^2.
	\end{gather*}
	We use the function constructed there, with $B(x, \infty) = 0$.
	
	Without loss of generality, we may assume that $\norm{a}_{\car} = 1$. We first consider real-valued $f$. Let $\dmu = u\dw$ and
	\[
		x_n = \EE_{\mu}\xkh{f \mid \cF_n} = \frac{\EE_n\xkh{f u}}{u_n}, \quad y_n = \EE_{\mu}\xkh{\sum_{k \geq n} a_k \mid \cF_n}.
	\]
	Then
	\[
		\EE_{\mu}\xkh{x_{n + 1} \mid \cF_n} = x_n, \quad \EE_{\mu}\xkh{y_{n + 1} \mid \cF_n} = y_n - a_n.
	\]
	Lemma~\ref{lem:carPleqMAX} gives
	\[
		y_n = \frac{\qw{u \sum_{k \geq n} a_k}{\cF_n}}{u_n} = \frac{\qw{\sum_{k \geq n} a_k u_k}{\cF_n}}{u_n} \leq \frac{u_n^M}{u_n} = \rho_n.
	\]
	By Jensen's inequality,
	\begin{align*}
		\EE_{\mu}\xkh{B(x_{n + 1}, y_{n + 1}) \mid \cF_n} - B\xkh{x_n, y_n} & \geq B\xkh{\EE_{\mu}\xkh{x_{n + 1} \mid \cF_n}, \EE_{\mu}\xkh{y_{n + 1} \mid \cF_n}} - B\xkh{x_n, y_n} \\
		& = B(x_n, y_n - a_n) - B\xkh{x_n, y_n} \\
		& \geq - \frac{\partial B}{\partial y}(x_n, y_n) a_n \\
		& \geq x_n^2\varphi(y_n) a_n \geq x_n^2\varphi(\rho_n) a_n. 
	\end{align*}
	The upper bound for $B$ gives
	\[
		0 \leq B(x_n, y_n) \leq 4 C_{\varphi} x_n^2 \leq 4 C_{\varphi} \EE_{\mu}\xkh{\abs{f}^2 \mid \cF_n}.
	\]
	Thus $B(x_n, y_n)$ is a uniformly integrable submartingale, and $x_n$ converges in $L^2(\mu)$ to $x_{\infty} = \EE_{\mu}\xkh{f \mid \cF_{\infty}}$, where $\cF_{\infty} = \bigvee_{n \in \NN}\cF_n$. Summing over $n \geq \ell$, we obtain
	\begin{align*}
		\EE_{\ell}\xkh{\sum_{n \geq \ell} a_n \frac{\abs{\EE_n\xkh{f u}}^2}{u_n} \varphi(\rho_n)} & = u_{\ell} \EE_{\mu}\xkh{\sum_{n \geq \ell} a_n x_n^2 \varphi(\rho_n) \mid \cF_{\ell}} \\
		& \leq u_{\ell} \EE_{\mu}\xkh{\sum_{n \geq \ell} \xkh{\EE_{\mu}\xkh{B(x_{n + 1}, y_{n + 1}) \mid \cF_n} - B(x_n, y_n)} \mid \cF_{\ell}} \\
		& = u_{\ell} \xkh{\EE_{\mu}\xkh{\lim_{n \to \infty} B(x_n, y_n) \mid \cF_{\ell}} - B(x_{\ell}, y_{\ell})} \\
		& \leq 4 C_{\varphi} u_{\ell} \EE_{\mu}\xkh{x_{\infty}^2 \mid \cF_{\ell}} \leq 4 C_{\varphi} \EE_{\ell}\xkh{\abs{f}^2 u}.
	\end{align*}
	This proves \eqref{eq:emb} in the normalized case. Scaling $a$ gives the factor $\norm{a}_{\car}$. 
\end{proof}

Taking $f = 1$ in \eqref{eq:emb}, we obtain the estimate
\begin{equation}
	\label{eq:emb-simple}
	\EE_{\ell}\xkh{\sum_{n \geq \ell} a_n u_n \varphi(\rho_n)} \leq 4 C_{\varphi} \norm{a}_{\car} u_{\ell},
\end{equation}
which will be used with $\varphi = \psi$.

\section{Sawyer-type Testing Conditions}
\label{sec:sawyer}

For sparse operators in the atomic setting, the proof in \cite{enma2016} uses Sawyer-type testing conditions
\[
	\int_{I} \abs{T(\chi_I u)}^2 v \dw \leq S \int_{I} u \dw, \qquad \int_{I} \abs{T(\chi_I v)}^2 u \dw \leq S \int_{I} v \dw,
\]
where $S < \infty$ is a constant. We will prove an analogous result in the general filtration setting, formulated in terms of what may be called a conditional Sawyer-type testing condition.

\begin{lem}
	\label{lem:sawyer}
	Let $a = \xkh{a_n}_{n \in \NN}$ be a nonnegative adapted sequence and let $T$, $T_{\geq n}$ be the operators defined by
	\[
		T f = \sum_{n \in \NN} a_n \EE_n\xkh{f}, \quad T_{\geq n} f = \sum_{k \geq n} a_k \EE_k\xkh{f}.
	\]
	
	The conditional Sawyer-type testing conditions mean that for every $n \in \NN$, the weights $u$, $v$ satisfy
	\[
		\norm{\sqrt{v} T_{\geq n} u}_{L^2}^{\cF_n} \leq S_u \sqrt{u_n}, \quad \norm{\sqrt{u} T_{\geq n} v}_{L^2}^{\cF_n} \leq S_v \sqrt{v_n}.
	\]
	Equivalently,
	\begin{align}
		\EE_{n}\xkh{v \xkh{\sum_{k \geq n} a_k u_k}^2} & \leq S_u^2 u_n, \\
		\EE_{n}\xkh{u \xkh{\sum_{k \geq n} a_k v_k}^2} & \leq S_v^2 v_n.
	\end{align}
	
	Then the operator $f \mapsto T(f u)$ is bounded from $L^2(u)$ to $L^2(v)$ if and only if $u$, $v$ satisfy the conditional Sawyer-type testing conditions.
\end{lem}

In this paper, we only use and prove the sufficiency of the conditional Sawyer-type testing conditions. The proof follows Treil \cite[Sections~3.2--3.5]{treil2015positive}, with stopping cubes replaced by principal sets. We use the principal set construction of Tanaka and Terasawa \cite[pp.~942--943]{tanaka-terasawa2013} in the form developed by Chen, Zhu, Zuo and Jiao \cite[Section~3]{chen2020}.

\proofpart Let $A_n \in \cF_n$, $\dmu = u \dw$ and $\EE_{n}^{u}\xkh{f} = \frac{\EE_{n}\xkh{f u}}{u_n}$. The construction of principal sets is as follows. We set
\[
	P_0^m = \dakh{2^{m - 1} < \EE_{0}^{u}\xkh{f} \leq 2^m}, \quad m \in \ZZ,
\]
and assume $\mu\xkh{P_0^m} > 0$. Let $\kappa_1\xkh{P_0^m} = 0$, $\kappa_2\xkh{P_0^m} = m$ and $\cP_0^m = \dakh{P_0^m}$. The zeroth generation of principal sets is $\cP_0^m$.

For $k \geq 0$, suppose we have constructed the $k$-th generation of principal sets $\cP_k^m$. Fixing $P \in \cP_k^m$, we define a stopping time
\[
	\tau_{P} = \infty \chi_{P^c} + \chi_P \inf\dakh{i > \kappa_1(P) : \omega \in A_i \text{ and } \EE_{i}^{u}\xkh{f} > 2^{\kappa_2(P) + 1}}.
\]
For each $i > \kappa_1(P)$, $j \geq \kappa_2(P) + 2$, we define
\[
	Q_{i, j}(P) = P \cap \dakh{\tau_P = i} \cap \dakh{2^{j - 1} < \EE_{i}^{u}\xkh{f} \leq 2^j}.
\]
Let
\[
	\cP(P) = \dakh{Q_{i, j}(P) : i > \kappa_1(P), j \geq \kappa_2(P) + 2, \mu\xkh{Q_{i, j}(P)} > 0}, \quad \cP_{k + 1}^m = \bigcup_{P \in \cP_k^m}{\cP(P)}.
\]
For $Q = Q_{i, j}(P)$, we let $\kappa_1(Q) = i = \tau_P$ and $\kappa_2(Q) = j$. Denote $E(P) = P \cap \dakh{\tau_P = \infty}$. We define the collection of principal sets $\cP$ by
\[
	\cP^m = \bigcup_{k \geq 0}{\cP_k^m}, \quad P_0 = \bigcup_{m \in \ZZ} P_0^m, \quad \cP_k = \bigcup_{m \in \ZZ} \cP_k^m, \quad \cP = \cP^{\ZZ} = \bigcup_{m \in \ZZ} \cP^m.
\]

\proofpart It is easy to see that the collection of principal sets $\cP^m$ satisfies the following properties:

\begin{enumerate}[label=(P\arabic*)]
	\item $P$ is $\cF_{\kappa_1(P)}$-measurable and the sets $Q \in \cP(P)$ are pairwise disjoint. \label{item:Pmeas&QdisUNION}
	\item We have $P_0^m = \bigsqcup_{P \in \cP^m}{E(P)}$ and $\chi_{P} \leq 2 \EE_{\kappa_1(P)}^{u}\xkh{\chi_{E(P)}} \chi_{P}$. \label{item:P0disUNION&2EP}
	\item For each $P \in \cP^m$, $2^{\kappa_2(P) - 1}\leq \EE_{\kappa_1(P)}^{u}\xkh{f} \leq 2^{\kappa_2(P)}$ holds on $P$. \label{item:domain2^ka}
	\item For every $n \geq 0$, we have $\chi_{P \cap \dakh{\kappa_1(P) \leq n < \tau_P}} \chi_{A_n} \EE_{n}^{u}\xkh{f} \leq 2^{\kappa_2(P) + 1} \chi_{P}$. \label{item:2^kaDOMAINinAn}
	\item There is a disjoint union decomposition
	\[
		A_n \cap \dakh{\EE_{0}^u\xkh{f} > 0} = \bigsqcup_{P \in \cP} A_n \cap P \cap \dakh{\kappa_1(P) \leq n < \tau_P}.
	\]
	\label{item:AnDISunion}
\end{enumerate}

The properties \ref{item:Pmeas&QdisUNION}--\ref{item:domain2^ka} are proved in \cite{chen2020}. To prove \ref{item:2^kaDOMAINinAn}, we consider $\omega \in P \cap \dakh{\kappa_1(P) \leq n < \tau_P} \cap A_n$. If $n = \kappa_1(P)$, then we have $\EE_{n}^{u}\xkh{f} \leq 2^{\kappa_2(P)} < 2^{\kappa_2(P) + 1}$ by \ref{item:domain2^ka}. For $n > \kappa_1(P)$, the inequality $\EE_{n}^{u}\xkh{f} > 2^{\kappa_2(P) + 1}$ would imply
\[
	\tau_P \leq n < \tau_P,
\]

a contradiction. Then we turn to proving \ref{item:AnDISunion}. Denote
\[
	R_{P, n} = A_n \cap P \cap \dakh{\kappa_1(P) \leq n < \tau_P}.
\]
Suppose $P, P' \in \cP^m$ and $P' \subsetneq P$. On $R_{P, n} \cap P'$, we have $\kappa_1(P') \geq \tau_{P} > n$, which implies $R_{P, n} \cap R_{P', n} = \varnothing$.

To verify the decomposition, we assume $\omega \in A_n \cap \dakh{\EE_{0}^u\xkh{f} > 0}$. There is a unique $m \in \ZZ$ such that $\omega \in P_0^m$ and $\kappa_1\xkh{P_0^m} = 0 \leq n$. Set $P = P_0^m$. If $n < \tau_{P}(\omega)$, then $\omega \in R_{P, n}$.

Otherwise, we set $i = \tau_{P}(\omega) \leq n$. There exists a principal set $Q_{i, j} \in \cP(P)$ such that $\omega \in Q_{i, j}$ and $\kappa_1(Q_{i, j}) = i > \kappa_1(P)$. Denote $P_0 = P$ and $P_1 = Q_{i, j}$. If $n \geq \tau_{P_1}(\omega)$, then the process continues and we can find a sequence of principal sets
\[
	P_N \subset \cdots \subset P_2 \subset P_1 \subset P_0, \quad n \geq \tau_{P_N}(\omega) > \cdots > \tau_{P_2}(\omega) > \tau_{P_1}(\omega) > \tau_{P_0}(\omega).
\]
This process must terminate after finitely many steps and $\tau_{P_{N + 1}} > n$. Thus we have $\omega \in R_{P_{N + 1}, n}$, which proves the decomposition property \ref{item:AnDISunion}.

\begin{proof}[Proof of Lemma~\ref{lem:sawyer}]
	\proofpart We will prove
	\[
		\sum_{n \in \NN} \int_{\Omega} a_n \EE_{n}\xkh{f u} \EE_{n}\xkh{g v} \dw \leq C (S_u + S_v) \norm{f}_{L^2(u)} \norm{g}_{L^2(v)},
	\]
	for some constant $C$ and for all $f \in L^2(u)$, $g \in L^2(v)$. This is equivalent to the boundedness of the operator $f \mapsto T(f u)$ from $L^2(u)$ to $L^2(v)$.
	
	Without loss of generality, we assume $f$ and $g$ are nonnegative and
	\[
		\norm{f}_{L^2(u)} = \norm{g}_{L^2(v)} = 1.
	\]
	Denote $\dmu = u \dw$, $\dnu = v \dw$ and
	\[
		\EE_{n}^{u}\xkh{f} = \frac{\EE_{n}\xkh{f u}}{u_n}, \quad \EE_{n}^{v}\xkh{g} = \frac{\EE_{n}\xkh{g v}}{v_n}.
	\]
	
	\proofpart Let
	\[
		\Lambda(f, g) = \sum_{n \in \NN} \int_{\Omega} a_n \EE_{n}\xkh{f u} \EE_{n}\xkh{g v} \dw,
	\]
	and
	\[
		A_n = \dakh{u_n \EE_{n}^{u}\xkh{f}^2 \geq v_n \EE_{n}^{v}\xkh{g}^2}, \quad B_n = \dakh{u_n \EE_{n}^{u}\xkh{f}^2 < v_n \EE_{n}^{v}\xkh{g}^2}.
	\]
	
	We consider the decomposition $\displaystyle \Lambda = \Lambda_A + \Lambda_B := \sum_{n \in \NN} \int_{A_n} \dots \dw + \sum_{n \in \NN} \int_{B_n} \dots \dw$ and want to prove $\Lambda_A \leq C S_u$. In the remainder of the proof, we consider only $\EE_{n}^{u}\xkh{f} > 0$.
	
	Using the above construction of principal sets, we define
	\[
		\Lambda_A^P = \sum_{n \in \NN} a_n \chi_{R_{P, n}} \EE_{n}\xkh{f u}.
	\]
	Then it is easy to see that
	\[
		\Lambda_A = \sum_{P \in \cP} \int_{P} \Lambda_A^P g \dnu.
	\]
	
	\proofpart To estimate $\Lambda_A^P$, \ref{item:2^kaDOMAINinAn} gives us
	\[
		\Lambda_A^P = \sum_{n \in \NN} a_n \chi_{R_{P, n}} \EE_{n}^{u}\xkh{f} u_n \leq \chi_{P} 2^{\kappa_2(P) + 1} \sum_{n \geq \kappa_1(P)} a_n u_n.
	\]
	Since $P \in \cF_{\kappa_1(P)}$, the conditional Sawyer-type testing condition gives us
	\begin{align*}
		\norm{\Lambda_A^P}_{L^2(v)}^{2} & \leq 2^{2(\kappa_2(P) + 1)} \int_P v \xkh{\sum_{n \geq \kappa_1(P)} a_n u_n}^2 \dw \\
		& = 2^{2(\kappa_2(P) + 1)} \int_P \EE_{\kappa_1(P)}\xkh{v \xkh{\sum_{n \geq \kappa_1(P)} a_n u_n}^2} \dw \\
		& \leq 2^{2(\kappa_2(P) + 1)} S_u^2 u(P).
	\end{align*}
	
	By Doob's maximal inequality, we have
	\[
		\int_{\Omega} \xkh{\sup_{n \in \NN}\abs{\EE_{n}^{u}\xkh{f}}}^2 \dmu \leq 4 \int_{\Omega} \abs{f}^2 \dmu.
	\]
	By \ref{item:P0disUNION&2EP} and \ref{item:domain2^ka}, we obtain
	\begin{align*}
		\sum_{P \in \cP} 2^{2 \kappa_{2}(P)} \chi_{P} & \leq \sum_{P \in \cP} \xkh{2 \EE_{\kappa_1(P)}^{u}\xkh{f}}^2 \xkh{2 \EE_{\kappa_1(P)}^{u} \xkh{\chi_{E(P)} \chi_{P}}} \\
		\sum_{P \in \cP} 2^{2 \kappa_{2}(P)} \mu(P) & \leq 8 \sum_{P \in \cP} \int_{E(P)} \EE_{\kappa_1(P)}^{u}\xkh{f}^2 \dmu \\
		& \leq 8 \int_{\Omega} \sup_{n \in \NN}\abs{\EE_{n}^{u}\xkh{f}}^2 \dmu \\
		& \leq 32 \int_{\Omega} \abs{f}^2 \dmu = 32.
	\end{align*}
	Thus we have
	\[
		\sum_{P \in \cP} \norm{\Lambda_A^P}_{L^2(v)}^{2} \leq 128 S_u^2.
	\]
	
	\proofpart Consider the disjoint decomposition
	\[
	P = E(P) \cup \bigsqcup_{Q \in \cP(P)} Q.
	\]
	On $E(P)$, the Cauchy-Schwarz inequality gives
	\[
		\sum_{P \in \cP} \int_{E(P)} \Lambda_A^P g \dnu \leq \xkh{\sum_{P \in \cP} \norm{\Lambda_A^P}_{L^2(v)}^2}^{\frac{1}{2}} \xkh{\sum_{P \in \cP} \int_{E(P)} g^2 \dnu}^{\frac{1}{2}} \leq 8 \sqrt{2} S_u.
	\]
	
	In the other case, since $Q \subset A_{\kappa_1(Q)}$, we can estimate the integral
	\begin{align*}
		\int_{Q} \xkh{\EE_{\kappa_1(Q)}^{v}\xkh{g}}^2 \dnu & = \int_{Q} \xkh{\EE_{\kappa_1(Q)}^{v}\xkh{g}}^2 v_{\kappa_1(Q)} \dw \\
		& \leq \int_Q \xkh{\EE_{\kappa_1(Q)}^{u}\xkh{f}}^2 u_{\kappa_1(Q)} \dw \\
		& = \int_Q \xkh{\EE_{\kappa_1(Q)}^{u}\xkh{f}}^2 \dmu \leq 2^{2 \kappa_2(Q)} \mu(Q),
	\end{align*}
	and
	\begin{align*}
		\sum_{P \in \cP} \sum_{Q \in \cP(P)} \int_Q \Lambda_A^P g \dnu & = \sum_{P \in \cP} \sum_{Q \in \cP(P)} \int_Q \Lambda_A^P \EE_{\kappa_1(Q)}^{v}\xkh{g} \dnu \\
		& \leq \xkh{\sum_{P \in \cP} \sum_{Q \in \cP(P)} \int_Q \xkh{\Lambda_A^P}^2 \dnu}^{\frac{1}{2}} \xkh{\sum_{P \in \cP} \sum_{Q \in \cP(P)} \int_Q \xkh{\EE_{\kappa_1(Q)}^{v}\xkh{g}}^2 \dnu} ^{\frac{1}{2}} \\
		& \leq \xkh{\sum_{P \in \cP} \norm{\Lambda_A^P}_{L^2(v)}^2}^{\frac{1}{2}} \xkh{\sum_{Q \in \cP} 2^{2 \kappa_2(Q)} \mu(Q)}^{\frac{1}{2}} \leq 64 S_u.
	\end{align*}
	
	All in all, we have $\Lambda_A \leq \xkh{64 + 8\sqrt{2}} S_u = C S_u$. By symmetry, we can also prove $\Lambda_B \leq C S_v$. This proves the lemma.
\end{proof}

By Lemma~\ref{lem:sawyer}, to prove Theorem~\ref{thm:1side}, it is enough to prove the testing estimate for $u$, since the other estimate follows by interchanging $u$ and $v$. Take $\psi$ as in \eqref{def:psi} and replace $v$ by $\frac{\alpha(1)}{A_u}v$, keeping the notation $v$. The separated bump condition and the lower bound for $\psi$ give
\[
	u_n v_n \leq \frac{\alpha(1)}{\rho_n\alpha(\rho_n)} \leq \psi(\rho_n).
\]
Thus, it remains to prove the following theorem. 

\begin{thm}
	\label{thm:testing}
	Let $a = \xkh{a_n}_{n \in \NN}$ be a Carleson sequence and let $T$ be defined by \eqref{def:CARoper}. Suppose the weights $u$, $v$ satisfy
	\begin{equation}
		\label{eq:wCON}
		u_n v_n \leq \psi\xkh{\rho_n},
	\end{equation}
	where $\psi$ is given by \eqref{def:psi}. Then, for every $\ell \in \NN$,
	\[
		\EE_{\ell}\xkh{v \xkh{\sum_{k \geq \ell} a_k u_k}^2} \leq C C_{\psi} \norm{a}_{\car}^2 u_\ell,
	\]
	where $C$ is an absolute constant.
\end{thm}

\section{Principal Sequence}
\label{sec:principal}

We first define principal sequences and prove estimates for their associated operators $T_b f = \sum_n b_n \EE_n\xkh{f}$. In Section~\ref{sec:testing}, we use these operators to dominate general Carleson operators $T_a$ and complete the proof of Theorem~\ref{thm:testing}.

Stopping-time formulations of sparse domination for martingales appear in \cite{domelevo-petermichl2016}. A pointwise passage from Carleson sums to sparse sums in the dyadic setting is given in \cite{conde-alonso-rey2016}. Here we use the conditional sparseness of the stopping times directly.

Fix $\ell \in \NN$. Let $(\tau_j)_{j \geq 0}$ be a sequence of stopping times taking values in $\{\ell, \ell + 1, \ldots\} \cup \{\infty\}$, such that $\tau_j < \tau_{j + 1}$ on $\dakh{\tau_j < \infty}$ and $\tau_{j + 1} = \infty$ on $\dakh{\tau_j = \infty}$. Suppose that
\[
	\EE_{\tau_j}\xkh{\chi_{\dakh{\tau_{j + 1} < \infty}}} \leq \frac{1}{2}
\]
on $\dakh{\tau_j < \infty}$. We define the principal sequence $b = \xkh{b_n}_{n \in \NN}$ by
\begin{equation}
	\label{def:principal-sequence}
	b_n = \sum_{j \geq 0} \chi_{\dakh{\tau_j = n}}, \quad n \geq \ell,
\end{equation}
and let $b_n = 0$ for $n < \ell$. The assumption on the stopping times can equivalently be written as
\[
	b_n \EE_n\xkh{\chi_{\dakh{\exists m > n : b_m = 1}}} \leq \frac{1}{2} b_n, \quad n \in \NN.
\]

Denote
\[
	E_n = \dakh{b_n = 1, b_k = 0 \text{ for all } k > n}.
\]
In applications, principal sequences and principal sets can be converted into each other. For a collection of principal sets $\cP$, we have
\[
	b_n = \sum_{\substack{P \in \cP \\ \kappa_1(P) = n}} \chi_P, \quad E_n = \bigsqcup_{\substack{P \in \cP \\ \kappa_1(P) = n}} E(P).
\]
Conversely, the sets $\dakh{\tau_0 = n_0, \dots, \tau_j = n_j}$ give the corresponding principal sets. It is easy to see that the principal sequence satisfies the following properties:

\begin{enumerate}[label=(B\arabic*)]
	\item $b_n$ is $\cF_n$-measurable and $b_n \in \dakh{0, 1}$. \label{item:bn-meas}
	\item We have $\dakh{\sum_{n \geq \ell} b_n > 0} = \bigsqcup_{n \geq \ell} E_n$ almost everywhere, and $b_n \leq 2 \EE_n\xkh{\chi_{E_n}}$. 
	\label{item:bn-sparse}
	\item $b$ is a Carleson sequence with $\norm{b}_{\car} \leq 2$. \label{item:bn-car}
\end{enumerate}

Property \ref{item:bn-meas} follows from the stopping time property. Since
\[
	E_n = \bigsqcup_{j \geq 0} \dakh{\tau_j = n, \tau_{j + 1} = \infty}, \qquad \dakh{\tau_j = n} \in \cF_n,
\]
we have
\begin{align*}
	\EE_n\xkh{\chi_{E_n}}
	& = \sum_{j \geq 0} \chi_{\dakh{\tau_j = n}} \xkh{1 - \EE_n\xkh{\chi_{\dakh{\tau_{j + 1} < \infty}}}} \\
	& \geq \frac{1}{2} \sum_{j \geq 0} \chi_{\dakh{\tau_j = n}}
	= \frac{1}{2} b_n.
\end{align*}

The sets $E_n$ are pairwise disjoint, so
\begin{align*}
	\EE_n\xkh{\sum_{k \geq n} b_k}
	& \leq 2 \sum_{k \geq n} \EE_n\xkh{\EE_k\xkh{\chi_{E_k}}} \\
	& = 2 \EE_n\xkh{\sum_{k \geq n} \chi_{E_k}}
	\leq 2,
\end{align*}
which proves \ref{item:bn-car}. In particular,
\[
	\sum_{n \geq \ell} b_n < \infty, \qquad \dakh{\sum_{n \geq \ell} b_n > 0} = \bigsqcup_{n \geq \ell} E_n
\]
almost everywhere, proving \ref{item:bn-sparse}.

In the special case where $b_{\ell} = 1$ almost everywhere, the set on the left-hand side is $\Omega$ up to a null set, and hence
\[
	\Omega = \bigsqcup_{n \geq \ell} E_n
\]
almost everywhere. 

If $c = \xkh{c_n}_{n \in \NN}$ is an adapted sequence with $c_n \in \dakh{0, 1}$ and $c_n \leq b_n$, then
\begin{align*}
	c_n \EE_n\xkh{\chi_{\dakh{\exists j > n : c_j = 1}}} & \leq c_n \EE_n\xkh{\chi_{\dakh{\exists j > n : b_j = 1}}} \\
	& \leq \frac{1}{2} c_n.
\end{align*}
Thus, $c$ is also a principal sequence.

\begin{cor}
	\label{cor:principal-embedding}
	Let $b$ be the principal sequence defined above. By Theorem~\ref{spEMB} with $a = b$, $f = 1$ and $\varphi = \psi$, we have
	\[
		\EE_{\ell}\xkh{\sum_{n \geq \ell} b_n u_n \psi(\rho_n)} \leq 8 C_{\psi} u_{\ell}.
	\]
\end{cor}

Suppose that each $\cF_n$ is generated by a family of atoms $\cD_n$. A principal sequence $b$ selects the family
\[
	\cQ = \bigcup_{n \geq \ell} \dakh{I \in \cD_n : b_n|_I = 1}.
\]
The conditional sparseness of $b$ gives
\[
	\abs{\bigcup_{\substack{J \in \cQ \\ J \subsetneq I}} J} \leq \frac{1}{2} \abs{I}, \quad I \in \cQ,
\]
so $\cQ$ is a sparse family. Conversely, a sparse family $\cQ \subset \bigcup_n \cD_n$ defines a principal sequence by
\[
	b_n = \sum_{I \in \cQ \cap \cD_n} \chi_I.
\]
Since conditional expectations are averages on the atoms, we have
\[
	T_b f = \sum_n b_n \EE_n\xkh{f} = \sum_{I \in \cQ} \lkh{f}_I \chi_I = T_{\cQ} f.
\]
In particular, this identifies principal sequences with sparse families in a dyadic filtration.

\subsection{Principal sequence estimation}

Conditional sparseness gives geometric decay for the number of selected times. We use this decay to estimate the weighted level sets of the corresponding sums. Geometrically decreasing generations are also used in local exponential estimates obtained from sparse decompositions; see \cite{lerner2010,ortiz-perez-rela2013}.

\begin{lem}
	\label{lem:principal-counting}
	Let $b$ be a principal sequence and let $\tau \geq \ell$ be a stopping time. Denote
	\[
		N = \sum_{n \geq \tau} b_n.
	\]
	Then, for every integer $k \geq 1$,
	\[
		\EE_{\tau}\xkh{\chi_{\dakh{N \geq k}}} \leq \frac{1}{2^{k - 1}} \EE_{\tau}\xkh{\chi_{\dakh{N \geq 1}}}
	\]
	on $\dakh{\tau < \infty}$.
\end{lem}

\begin{proof}
	Let
	\[
		S_n = \sum_{\tau \leq j \leq n} b_j.
	\]
	The case $k = 1$ is immediate. For every $k \geq 2$, we have
	\begin{align*}
		\chi_{\dakh{N \geq k}} & = \sum_{n \geq \tau} b_n \chi_{\dakh{S_n = k}}, \\
		& = \sum_{n \geq \tau} b_n \chi_{\dakh{S_n = k - 1}} \chi_{\dakh{\exists j > n : b_j = 1}}.
	\end{align*}
	Thus,
	\begin{align*}
		\EE_{\tau}\xkh{\chi_{\dakh{N \geq k}}} & = \EE_{\tau}\xkh{\sum_{n \geq \tau} b_n \chi_{\dakh{S_n = k - 1}} \chi_{\dakh{\exists j > n : b_j = 1}}} \\
		& = \EE_{\tau}\xkh{\sum_{n \geq \tau} b_n \chi_{\dakh{S_n = k - 1}} \EE_n\xkh{\chi_{\dakh{\exists j > n : b_j = 1}}}} \\
		& \leq \frac{1}{2} \EE_{\tau}\xkh{\sum_{n \geq \tau} b_n \chi_{\dakh{S_n = k - 1}}} \\
		& = \frac{1}{2} \EE_{\tau}\xkh{\chi_{\dakh{N \geq k - 1}}} \leq \frac{1}{2^{k - 1}} \EE_{\tau}\xkh{\chi_{\dakh{N \geq 1}}}. \qedhere
	\end{align*}
\end{proof}

The weighted form is as follows.

\begin{cor}
	\label{cor:principal-weighted-counting}
	Under the assumptions of Lemma~\ref{lem:principal-counting}, let $u, v$ be weights and let $a, d$ be $\cF_{\tau}$-measurable functions with $a > 0$, $d \geq 0$. Suppose that, for every $n \geq \ell$,
	\[
		u_n \geq a, \quad v_n \leq d \quad \text{on } \dakh{\tau \leq n, b_n = 1}.
	\]
	Then, for every integer $k \geq 1$,
	\[
		\EE_{\tau}\xkh{v \chi_{\dakh{N \geq k}}} \leq \frac{1}{2^{k - 1}} \frac{d}{a} u_{\tau}
	\]
	on $\dakh{\tau < \infty}$.
\end{cor}

This gives the following weak $L^2$ estimate.

\begin{lem}
	\label{lem:principal-weak}
	Let $b$ be a principal sequence and let $\tau \geq \ell$ be a stopping time. Let $q$ be a bounded nonnegative $\cF_{\tau}$-measurable function. Suppose the weights $u$, $v$ satisfy, for every $n \geq \ell$,
	\[
		u_n v_n \leq q \quad \text{on } \dakh{\tau \leq n, b_n = 1}.
	\]
	Denote
	\[
		H = \sum_{n \geq \tau} b_n u_n.
	\]
	Then, for every positive $\cF_{\tau}$-measurable function $t$,
	\[
		t^2 \EE_{\tau}\xkh{v \chi_{\dakh{H > t}}} \leq 288 q u_{\tau}
	\]
	on $\dakh{\tau < \infty}$.
\end{lem}

\begin{proof}
	Let $h$ be a positive $\cF_{\tau}$-measurable function. For $n \geq \tau$, define
	\begin{align*}
		c_n & = b_n \chi_{\dakh{u_n > h}}, \\
		c_n^{(m)} & = b_n \chi_{\dakh{2^{-m - 1}h < u_n \leq 2^{-m}h}}, \quad m \geq 0,
	\end{align*}
	and set $c_n = c_n^{(m)} = 0$ for $n < \tau$. Denote
	\[
		N = \sum_{n \geq \tau} c_n, \quad N_m = \sum_{n \geq \tau} c_n^{(m)}.
	\]
	Since $c_n \leq b_n$ and $c_n^{(m)} \leq b_n$, these are principal sequences. We have
	\begin{align*}
		u_n > h, \quad v_n & \leq \frac{q}{h} && \text{on } \dakh{c_n = 1}, \\
		u_n > 2^{-m - 1}h, \quad v_n & \leq \frac{q}{2^{-m - 1}h} && \text{on } \dakh{c_n^{(m)} = 1}.
	\end{align*}
	By Corollary~\ref{cor:principal-weighted-counting},
	\begin{align*}
		\EE_{\tau}\xkh{v \chi_{\dakh{N \geq 1}}} & \leq \frac{q}{h^2} u_{\tau}, \\
		\EE_{\tau}\xkh{v \chi_{\dakh{N_m \geq k}}} & \leq \frac{1}{2^{k - 1}} \frac{q}{\xkh{2^{-m - 1}h}^2} u_{\tau} = 2^{2m + 3 - k} \frac{q}{h^2} u_{\tau}, \\
		\EE_{\tau}\xkh{v \chi_{\dakh{N_m \geq 3m + 4}}} & \leq 2^{-m - 1} \frac{q}{h^2} u_{\tau}, \quad \text{take } k = 3m + 4. 
	\end{align*}
	On $\dakh{N = 0} \cap \bigcap_{m \geq 0} \dakh{N_m < 3m + 4}$,
	\begin{align*}
		H & = \sum_{m \geq 0} \sum_{n \geq \tau} c_n^{(m)} u_n \\
		& \leq h \sum_{m \geq 0} 2^{-m} N_m \\
		& \leq h \sum_{m \geq 0} 2^{-m}(3m + 3) = 12h.
	\end{align*}
	Taking complements in $\dakh{N = 0} \cap \bigcap_{m \geq 0} \dakh{N_m < 3m + 4} \subset \dakh{H \leq 12 h}$ gives
	\[
		\dakh{H > 12h} \subseteq \dakh{N \geq 1} \cup \bigcup_{m \geq 0} \dakh{N_m \geq 3m + 4}.
	\]
	Using the estimates above, we obtain
	\begin{align*}
		\EE_{\tau}\xkh{v \chi_{\dakh{H > 12h}}} & \leq \EE_{\tau}\xkh{v \chi_{\dakh{N \geq 1}}} + \sum_{m \geq 0} \EE_{\tau}\xkh{v \chi_{\dakh{N_m \geq 3m + 4}}} \\
		& \leq \xkh{1 + \sum_{m \geq 0} 2^{-m - 1}} \frac{q}{h^2} u_{\tau} = \frac{2q}{h^2} u_{\tau}.
	\end{align*}
	Taking $h = \frac{t}{12}$ completes the proof.
\end{proof}

\subsection{A \texorpdfstring{good-$\lambda$-type}{good-lambda-type} estimate}

We next compare two levels of the sum when its largest summand is small, and apply the preceding weak estimate between the two level crossings. Such comparisons are used in the good-$\lambda$ method for martingales \cite{burkholder-gundy1970} and weighted singular integrals \cite{coifman-fefferman1974}. Local decay estimates and related good-$\lambda$ inequalities are developed in \cite{buckley1993,karagulyan2021}.

\begin{lem}
	\label{lem:good-lambda}
	Assume \eqref{eq:wCON}. Let $b$ be a principal sequence and denote
	\[
		H = \sum_{n \geq \ell} b_n u_n, \quad s_n = \sum_{\ell \leq j \leq n} b_j u_j, \quad s_{\ell - 1} = 0, \quad L = \sup_{n \geq \ell} b_n u_n.
	\]
	For $\lambda > 0$, let
	\[
		\gamma_0 = \inf\dakh{n \geq \ell : s_{n - 1} \leq \lambda < s_n, \, u_n \leq \frac{\lambda}{2}}.
	\]
	Then
	\begin{equation}
		\label{eq:good-lambda}
		\lambda^2 \EE_{\ell}\xkh{v \chi_{\dakh{H > 4\lambda, \, L \leq \frac{\lambda}{2}}}} \lesssim \EE_{\ell}\xkh{\chi_{\dakh{\gamma_0 < \infty}} u_{\gamma_0} \psi(\rho_{\gamma_0})} + \EE_{\ell}\xkh{\sum_{n \geq \ell} u_n e_n \chi_{\dakh{\lambda < s_n \leq 4\lambda}}},
	\end{equation}
\end{lem}

\begin{proof}
	\proofpart Let
	\[
		\tau = \inf\dakh{n \geq \ell : s_n > 4\lambda}.
	\]
	By the definitions of $\gamma_0$ and $\tau$, we have
	\begin{gather*}
		\lambda < s_{\gamma_0} \leq \lambda + u_{\gamma_0} \leq \frac{3\lambda}{2}, \\
		\lambda < s_n \leq 4\lambda, \qquad \text{for } \gamma_0 \leq n < \tau.
	\end{gather*}
	On $\dakh{H > 4\lambda, \, L \leq \frac{\lambda}{2}}$, the sum between these stopping times satisfies
	\begin{align*}
		\sum_{\gamma_0 \leq n < \tau} b_n u_n & = s_{\tau} - u_{\tau} - s_{\gamma_0 - 1} \\
		& > 4\lambda - \frac{\lambda}{2} - \lambda = \frac{5\lambda}{2}.
	\end{align*}
	
	\proofpart For $m \geq 0$, define
	\[
		\gamma_{m + 1} = \inf\dakh{n > \gamma_m : b_n = 1, n < \tau, \psi(\rho_n) > 2 \psi\xkh{\rho_{\gamma_m}}},
	\]
	and let
	\[
		c_n^{(m)} = b_n \chi_{\dakh{\gamma_m \leq n < \gamma_{m + 1} \wedge \tau}}, \qquad T_m^{\lambda} = \sum_{n \geq \ell} c_n^{(m)} u_n.
	\]
	These intervals partition $[\gamma_0, \tau)$, so the preceding estimate gives
	\begin{align*}
		\sum_{m \geq 0} T_m^{\lambda} & = \sum_{\gamma_0 \leq n < \tau} b_n u_n, \\
		\dakh{H > 4\lambda, L \leq \frac{\lambda}{2}} & \subseteq \dakh{\sum_{m \geq 0} T_m^{\lambda} > \frac{5\lambda}{2}}.
	\end{align*}
	Since $c^{(m)}$ is a principal sequence and
	\[
		u_n v_n \leq \psi(\rho_n) \leq 2 \psi\xkh{\rho_{\gamma_m}} \quad \text{on } \dakh{c_n^{(m)} = 1},
	\]
	Lemma~\ref{lem:principal-weak} gives, for every positive $\cF_{\gamma_m}$-measurable function $t$,
	\begin{equation}
		\label{eq:block-weak}
		t^2 \EE_{\gamma_m}\xkh{v \chi_{\dakh{T_m^{\lambda} > t}}} \leq 288 \cdot 2 \psi\xkh{\rho_{\gamma_m}} u_{\gamma_m} = 576 \cdot u_{\gamma_m} \psi\xkh{\rho_{\gamma_m}}.
	\end{equation}
	
	\proofpart Let $t_m$ be positive $\cF_{\gamma_m}$-measurable functions such that
	\[
		\sum_{\substack{m \geq 0 \\ \gamma_m < \infty}} t_m \leq \lambda.
	\]
	This gives the following inclusions:
	\begin{align*}
		\dakh{H > 4\lambda, L \leq \frac{\lambda}{2}}
		& \subseteq \dakh{\sum_{\substack{m \geq 0 \\ \gamma_m < \infty}} T_m^{\lambda} > \frac{5\lambda}{2}} \\
		& \subseteq \dakh{\sum_{\substack{m \geq 0 \\ \gamma_m < \infty}} T_m^{\lambda} > \sum_{\substack{m \geq 0 \\ \gamma_m < \infty}} t_m} \\
		& \subseteq \bigcup_{m \geq 0} \dakh{\gamma_m < \infty, T_m^{\lambda} > t_m}.
	\end{align*}
	By \eqref{eq:block-weak},
	\begin{align*}
		\lambda^2 \EE_{\ell}\xkh{v \chi_{\dakh{H > 4\lambda, L \leq \frac{\lambda}{2}}}} & \leq \lambda^2 \EE_{\ell}\xkh{v \chi_{\bigcup_{m \geq 0} \dakh{\gamma_m < \infty, T_m^{\lambda} > t_m}}} \\
		& \leq \EE_{\ell}\xkh{\sum_{m \geq 0} \chi_{\dakh{\gamma_m < \infty}} \lambda^2 v \chi_{\dakh{T_m^{\lambda} > t_m}}} \\
		& = \EE_{\ell}\xkh{\sum_{m \geq 0} \chi_{\dakh{\gamma_m < \infty}} \frac{\lambda^2}{t_m^2} t_m^2 \EE_{\gamma_m}\xkh{v \chi_{\dakh{T_m^{\lambda} > t_m}}}} \\
		& \leq 576 \cdot \EE_{\ell}\xkh{\sum_{m \geq 0} \chi_{\dakh{\gamma_m < \infty}} \frac{\lambda^2}{t_m^2} u_{\gamma_m} \psi\xkh{\rho_{\gamma_m}}}.
	\end{align*}
	
	\proofpart To apply Corollary~\ref{cor:stopped-defect} to the preceding sum, we choose $t_m$ so that $\frac{\lambda^2}{t_m^2} \approx \rho_{\gamma_m}$ for $m \geq 1$. By the stopping rule and the scaling property of $\psi$, we have
	\begin{align*}
		\psi\xkh{\rho_{\gamma_{m + 1}}} & > 2 \psi\xkh{\rho_{\gamma_m}} \\
		& \geq 2 \psi\xkh{\rho_{\gamma_{m + 1}}} \frac{\rho_{\gamma_{m + 1}}^2}{\rho_{\gamma_m}^2}, \qquad \gamma_{m + 1} < \infty.
	\end{align*}
	Thus,
	\[
		\rho_{\gamma_{m + 1}} < 2^{-\frac{1}{2}} \rho_{\gamma_m}, \qquad \gamma_{m + 1} < \infty.
	\]
	On $\dakh{\gamma_1 < \infty}$, the preceding inequality gives
	\[
		1 \leq \rho_{\gamma_m} \leq 2^{-\frac{m - 1}{2}} \rho_{\gamma_1}, \qquad m \geq 1, \quad \gamma_m < \infty.
	\]
	Hence there is a largest index $M$ such that $\gamma_M < \infty$. For $1 \leq m \leq M$, we have
	\[
		\rho_{\gamma_m}^{-\frac{1}{2}} \leq 2^{-\frac{M - m}{4}} \rho_{\gamma_M}^{-\frac{1}{2}} \leq 2^{-\frac{M - m}{4}}.
	\]
	Therefore,
	\begin{align*}
		\sum_{\substack{m \geq 1 \\ \gamma_m < \infty}} \rho_{\gamma_m}^{-\frac{1}{2}} & = \sum_{m = 1}^M \rho_{\gamma_m}^{-\frac{1}{2}} \\
		& \leq \sum_{m = 1}^M 2^{-\frac{M - m}{4}} \\
		& = \sum_{j = 0}^{M - 1} 2^{-\frac{j}{4}} \leq \frac{1}{1 - 2^{-\frac{1}{4}}}.
	\end{align*}
	We take\footnote{We treat $m = 0$ separately because the estimate below for $m \geq 1$ uses $\EE_{\gamma_{m - 1}}$.}
	\[
		c = \frac{1 - 2^{-\frac{1}{4}}}{2}, \qquad t_0 = \frac{\lambda}{2}, \qquad t_m = c\lambda \rho_{\gamma_m}^{-\frac{1}{2}}, \quad m \geq 1.
	\]
	Then
	\begin{gather*}
		t_0 + \sum_{\substack{m \geq 1 \\ \gamma_m < \infty}} t_m \leq \frac{\lambda}{2} + \frac{c\lambda}{1 - 2^{-\frac{1}{4}}} = \lambda, \\
		\frac{\lambda^2}{t_0^2} = 4, \qquad \frac{\lambda^2}{t_m^2} = c^{-2} \rho_{\gamma_m}, \quad m \geq 1.
	\end{gather*}
	
	\proofpart For $m \geq 1$, applying Corollary~\ref{cor:stopped-defect} between $\gamma_{m - 1}$ and $\gamma_m \wedge \tau$, we have
	\[
		\begin{aligned}
			& \EE_{\gamma_{m - 1}}\xkh{\chi_{\dakh{\gamma_m < \infty}} u_{\gamma_m} \rho_{\gamma_m} \psi\xkh{\rho_{\gamma_m}}} \\
			& \quad \leq 48 \cdot \EE_{\gamma_{m - 1}}\xkh{\sum_{\gamma_{m - 1} \leq n < \gamma_m \wedge \tau} u_n e_n}.
		\end{aligned}
	\]
	The stopping intervals are disjoint and lie in $[\gamma_0, \tau)$, where $\lambda < s_n \leq 4\lambda$. Hence,
	\[
		\sum_{m \geq 1} \chi_{\dakh{\gamma_{m - 1} \leq n < \gamma_m \wedge \tau}} \leq \chi_{\dakh{\lambda < s_n \leq 4\lambda}}.
	\]
	Taking $\EE_{\ell}$ in the stopped estimate and summing over $m \geq 1$, we obtain
	\begin{align*}
		\EE_{\ell}\xkh{\sum_{m \geq 1} \chi_{\dakh{\gamma_m < \infty}} u_{\gamma_m} \rho_{\gamma_m} \psi\xkh{\rho_{\gamma_m}}} & \leq 48 \cdot \EE_{\ell}\xkh{\sum_{m \geq 1} \sum_{\gamma_{m - 1} \leq n < \gamma_m \wedge \tau} u_n e_n} \\
		& = 48 \cdot \EE_{\ell}\xkh{\sum_{n \geq \ell} u_n e_n \sum_{m \geq 1} \chi_{\dakh{\gamma_{m - 1} \leq n < \gamma_m \wedge \tau}}} \\
		& \leq 48 \cdot \EE_{\ell}\xkh{\sum_{n \geq \ell} u_n e_n \chi_{\dakh{\lambda < s_n \leq 4\lambda}}}.
	\end{align*}
	
	\proofpart Combining the above estimates, we obtain
	\begin{align*}
		\lambda^2 \EE_{\ell}\xkh{v \chi_{\dakh{H > 4\lambda, L \leq \frac{\lambda}{2}}}} & \leq 576 \EE_{\ell}\xkh{\sum_{m \geq 0} \chi_{\dakh{\gamma_m < \infty}} \frac{\lambda^2}{t_m^2} u_{\gamma_m} \psi\xkh{\rho_{\gamma_m}}} \\
		& \leq 576 \xkh{4 \EE_{\ell}\xkh{\chi_{\dakh{\gamma_0 < \infty}} u_{\gamma_0} \psi(\rho_{\gamma_0})} + c^{-2} \EE_{\ell}\xkh{\sum_{m \geq 1} \chi_{\dakh{\gamma_m < \infty}} u_{\gamma_m} \rho_{\gamma_m} \psi\xkh{\rho_{\gamma_m}}}} \\
		& \lesssim \EE_{\ell}\xkh{\chi_{\dakh{\gamma_0 < \infty}} u_{\gamma_0} \psi(\rho_{\gamma_0})} + \EE_{\ell}\xkh{\sum_{n \geq \ell} u_n e_n \chi_{\dakh{\lambda < s_n \leq 4\lambda}}}. \qedhere
	\end{align*}
\end{proof}

\section{Proof of the Main Result}
\label{sec:testing}

We integrate \eqref{eq:good-lambda} with respect to $\lambda$ to obtain the square estimate for principal sequences.

\begin{prop}
	\label{prop:principal-square}
	Assume \eqref{eq:wCON}. For every principal sequence $b$ and $\ell \in \NN$,
	\[
		\EE_{\ell}\xkh{v \xkh{\sum_{n \geq \ell} b_n u_n}^2} \lesssim C_{\psi} u_{\ell}.
	\]
\end{prop}

\begin{proof}
	Use the notation of Lemma~\ref{lem:good-lambda}, writing $\gamma_0 = \gamma_0(\lambda)$. We note that $b_{\gamma_0} = 1$. Fix $n \geq \ell$. By the definition of $\gamma_0$,
	\begin{align*}
		\int_0^{\infty} \chi_{\dakh{\gamma_0(\lambda) = n}} \frac{\dd{\lambda}}{\lambda}
		& \leq b_n \int_{[s_{n - 1}, s_n) \cap [2u_n, \infty)} \frac{\dd{\lambda}}{\lambda} \\
		& \leq \frac{b_n(s_n - s_{n - 1})}{2u_n} = \frac{b_n}{2}.
	\end{align*}
	Summing over $n \geq \ell$ and taking $\EE_{\ell}$, we have
	\begin{align*}
		\int_0^{\infty} \EE_{\ell}\xkh{\chi_{\dakh{\gamma_0(\lambda) < \infty}} u_{\gamma_0(\lambda)} \psi(\rho_{\gamma_0(\lambda)})} \frac{\dd{\lambda}}{\lambda} & = \EE_{\ell}\xkh{\sum_{n \geq \ell} u_n \psi(\rho_n) \int_0^{\infty} \chi_{\dakh{\gamma_0(\lambda) = n}} \frac{\dd{\lambda}}{\lambda}} \\
		& \leq \frac{1}{2} \EE_{\ell}\xkh{\sum_{n \geq \ell} b_n u_n \psi(\rho_n)}.
	\end{align*}
	For the other term, using $\displaystyle \int_{\frac{s_n}{4}}^{s_n} \frac{\dd{\lambda}}{\lambda} = \log 4$, we obtain
	\begin{align*}
		\int_0^{\infty} \EE_{\ell}\xkh{\sum_{n \geq \ell} u_n e_n \chi_{\dakh{\lambda < s_n \leq 4\lambda}}} \frac{\dd{\lambda}}{\lambda} & = \EE_{\ell}\xkh{\sum_{n \geq \ell} u_n e_n \int_0^{\infty} \chi_{\dakh{\lambda < s_n \leq 4\lambda}} \frac{\dd{\lambda}}{\lambda}} \\
		& = \EE_{\ell}\xkh{\sum_{n \geq \ell} u_n e_n \int_0^{\infty} \chi_{\dakh{\frac{s_n}{4} \leq \lambda \leq s_n}} \frac{\dd{\lambda}}{\lambda}} \\
		& \leq \log 4 \cdot \EE_{\ell}\xkh{\sum_{n \geq \ell} u_n e_n}.
	\end{align*}
	Thus, integrating \eqref{eq:good-lambda} against $\displaystyle 2 \frac{\dd{\lambda}}{\lambda}$, we obtain
	\begin{equation}
		\label{eq:integrated-good-lambda}
		\int_0^{\infty} 2\lambda \EE_{\ell}\xkh{v \chi_{\dakh{H > 4\lambda, L \leq \frac{\lambda}{2}}}} \dd{\lambda} \lesssim \EE_{\ell}\xkh{\sum_{n \geq \ell} b_n u_n \psi(\rho_n)} + \EE_{\ell}\xkh{\sum_{n \geq \ell} u_n e_n}.
	\end{equation}
	The remaining part is controlled by $L$. Since $b_n \in \dakh{0, 1}$, condition \eqref{eq:wCON} gives
	\begin{align*}
		\EE_{\ell}\xkh{v L^2} & \leq \EE_{\ell}\xkh{\sum_{n \geq \ell} v b_n u_n^2} \\
		& = \EE_{\ell}\xkh{\sum_{n \geq \ell} b_n u_n^2 v_n} \\
		& \leq \EE_{\ell}\xkh{\sum_{n \geq \ell} b_n u_n \psi(\rho_n)}.
	\end{align*}
	Applying the conditional layer-cake formula and then \eqref{eq:integrated-good-lambda}, we have
	\begin{align*}
		\EE_{\ell}\xkh{v H^2} & = 16 \int_0^{\infty} 2\lambda \EE_{\ell}\xkh{v \chi_{\dakh{H > 4\lambda}}} \dd{\lambda} \\
		& \leq 16 \int_0^{\infty} 2\lambda \EE_{\ell}\xkh{v \chi_{\dakh{H > 4\lambda, L \leq \frac{\lambda}{2}}}} \dd{\lambda} + 16 \int_0^{\infty} 2\lambda \EE_{\ell}\xkh{v \chi_{\dakh{L > \frac{\lambda}{2}}}} \dd{\lambda} \\
		& = 16 \int_0^{\infty} 2\lambda \EE_{\ell}\xkh{v \chi_{\dakh{H > 4\lambda, L \leq \frac{\lambda}{2}}}} \dd{\lambda} + 64 \EE_{\ell}\xkh{v L^2} \\
		& \lesssim \EE_{\ell}\xkh{\sum_{n \geq \ell} b_n u_n \psi(\rho_n)} + \EE_{\ell}\xkh{\sum_{n \geq \ell} u_n e_n} + \EE_{\ell}\xkh{\sum_{n \geq \ell} b_n u_n \psi(\rho_n)} \lesssim C_{\psi} u_{\ell}.
	\end{align*}
	The last inequality follows from Corollary~\ref{cor:principal-embedding} and \eqref{eq:defect-total}.
\end{proof}

We now construct a principal sequence to dominate the sum in Theorem~\ref{thm:testing}.

\begin{proof}[Proof of Theorem~\ref{thm:testing}]
	Fix $\ell \in \NN$ and put $K = \norm{a}_{\car}$. Let $\tau_0 = \ell$ and define
	\[
		\tau_{j + 1} = \inf\dakh{k > \tau_j : u_k > 4u_{\tau_j} \text{ or } \sum_{i = \tau_j}^k a_i > 4K}.
	\]
	By the conditional Doob inequality and the Carleson condition,
	\begin{align*}
		\EE_{\tau_j}\xkh{\chi_{\dakh{\tau_{j + 1} < \infty}}}
		& \leq \EE_{\tau_j}\xkh{\chi_{\dakh{\sup_{k > \tau_j} u_k > 4u_{\tau_j}}}} + \EE_{\tau_j}\xkh{\chi_{\dakh{\sum_{i \geq \tau_j} a_i > 4K}}} \\
		& \leq \frac{1}{4} + \frac{1}{4K} \EE_{\tau_j}\xkh{\sum_{i \geq \tau_j} a_i} \leq \frac{1}{2}.
	\end{align*}
	Hence $b_n = \sum_{j \geq 0} \chi_{\dakh{\tau_j = n}}$ is a principal sequence. Since $a_n \leq K$, the definition of $\tau_j$ gives
	\[
		u_n \leq 4u_{\tau_j}, \qquad \sum_{\tau_j \leq i \leq n} a_i \leq 4K, \qquad \text{for } \tau_j \leq n < \tau_{j + 1}.
	\]
	Summing over the stopping intervals, we obtain
	\begin{align*}
		\sum_{n \geq \ell} a_n u_n & = \sum_{j \geq 0} \sum_{\tau_j \leq n < \tau_{j + 1}} a_n u_n \\
		& \leq 4 \sum_{j \geq 0} u_{\tau_j} \sum_{\tau_j \leq n < \tau_{j + 1}} a_n \\
		& \leq 16K \sum_{j \geq 0} u_{\tau_j} = 16K \sum_{n \geq \ell} b_n u_n.
	\end{align*}
	Proposition~\ref{prop:principal-square} now yields
	\[
		\EE_{\ell}\xkh{v \xkh{\sum_{n \geq \ell} a_n u_n}^2} \leq (16K)^2 \EE_{\ell}\xkh{v \xkh{\sum_{n \geq \ell} b_n u_n}^2} \lesssim C_{\psi} K^2 u_{\ell}. \qedhere
	\]
\end{proof}

\begin{proof}[Proof of Theorem~\ref{thm:1side}]
	Apply Theorem~\ref{thm:testing} to $\widetilde{v} = \frac{\alpha(1)}{A_u}v$. Recalling that $C_{\psi} = 2\alpha(1)C_{\alpha}$, we obtain
	\begin{align*}
		\EE_{\ell}\xkh{v \xkh{\sum_{n \geq \ell} a_n u_n}^2}
		& = \frac{A_u}{\alpha(1)} \EE_{\ell}\xkh{\widetilde{v} \xkh{\sum_{n \geq \ell} a_n u_n}^2} \\
		& \lesssim \frac{A_u}{\alpha(1)} C_{\psi} \norm{a}_{\car}^2 u_{\ell} \\
		& = 2 C_{\alpha} A_u \norm{a}_{\car}^2 u_{\ell}.
	\end{align*}
	Interchanging $u$ and $v$ gives the other testing estimate, and Lemma~\ref{lem:sawyer} completes the proof.
\end{proof}

\begin{rem}
	\label{rem:sparse-cz}
	By the correspondence between principal sequences and sparse families described in Section~\ref{sec:principal}, Proposition~\ref{prop:principal-square} and Lemma~\ref{lem:sawyer} give the separated entropy bump theorem for sparse operators. Lemma~\ref{lem:YoungAlpha} then gives the separated Orlicz bump theorem under the assumptions of Theorem~\ref{thm:1side-bump}.
	
	The connection with Calder\'on--Zygmund operators follows from Lerner's theorem \cite[Theorem~1.1]{lerner2013}. For a Calder\'on--Zygmund operator $T$ on $\RR^d$ and a Banach function space $X$ over $\RR^d$, one has
	\[
		\norm{T_{\natural} f}_X \leq C(T, d) \sup_{\cD, \, \cQ \subset \cD} \norm{T_{\cQ}\abs{f}}_X,
	\]
	where the supremum is taken over all dyadic lattices $\cD$ and all sparse families $\cQ \subset \cD$, and $T_{\natural}$ is the maximal truncated operator
	\[
		T_{\natural} f(x) = \sup_{0 < \varepsilon < R < \infty} \abs{\int_{\varepsilon < \abs{x - y} < R} K(x, y) f(y) \dd{y}}.
	\]
	If the corresponding bump bounds are uniform over the dyadic lattices, taking $X = L^2(v)$ gives
	\[
		\norm{T_{\natural}(f u)}_{L^2(v)} \lesssim \sup_{\cD, \, \cQ \subset \cD} \norm{T_{\cQ}(\abs{f} u)}_{L^2(v)} \lesssim \norm{f}_{L^2(u)}.
	\]
	Hence the corresponding separated entropy and Orlicz bump conditions also imply the two-weight estimates for Calder\'on--Zygmund operators.
\end{rem}

\section{Further Remarks}
\label{sec:further}

The method is not restricted to entropy bumps or to the exponent $2$. We prove an Orlicz maximal estimate and a Carleson embedding, then outline the remaining changes, which require no new proof ideas.

\proofpart The separated Orlicz bump estimate can also be proved directly, without passing through entropy bumps. For a finite Young function $\Phi$, normalized by $\Phi(1) = 1$, put
\[
	u_n^\Phi = \essinf_{\lambda > 0} \lambda\xkh{1 + \EE_n\xkh{\Phi\xkh{\frac{u}{\lambda}}}}, \qquad \rho_n^\Phi = \frac{u_n^\Phi}{u_n},
\]
where $\lambda$ ranges over positive $\cF_n$-measurable functions. This is the conditional Amemiya norm, and
\[
	\norm{u}^{\cF_n}_{L^\Phi} \leq u_n^\Phi \leq 2 \norm{u}^{\cF_n}_{L^\Phi}, \qquad \EE_n^u\xkh{\rho_{n + 1}^\Phi} \leq \rho_n^\Phi.
\]
We first prove an estimate for $\sup_{n \geq \ell} \norm{h}^{\cF_n}_{L^\Psi}$.

\begin{prop}
	\label{prop:orlicz-maximal}
	Let $\Psi$ be a finite Young function such that
	\[
		\int_1^{\infty} \frac{\Psi(t)}{t^3} \dt < \infty.
	\]
	For $h \in L^2$ and $\ell \in \NN$, put
	\[
		M_{\Psi, \geq \ell} h = \sup_{n \geq \ell} \norm{h}^{\cF_n}_{L^\Psi}.
	\]
	Then
	\[
		\EE_{\ell}\xkh{\xkh{M_{\Psi, \geq \ell} h}^2} \lesssim_{\Psi} \EE_{\ell}\xkh{\abs{h}^2}.
	\]
\end{prop}

\begin{proof}
	Choose $c > 0$ with $\Psi(c) \leq \frac{1}{2}$. For $t > 0$, let
	\[
		\tau = \inf\dakh{n \geq \ell : \norm{h}^{\cF_n}_{L^\Psi} > t}.
	\]
	On $\dakh{\tau < \infty}$, we have
	\begin{align*}
		1 < \EE_{\tau}\xkh{\Psi\xkh{\frac{\abs{h}}{t}}} & = \EE_{\tau}\xkh{\Psi\xkh{\frac{\abs{h}}{t}} \chi_{\dakh{\abs{h} \leq ct}}} + \EE_{\tau}\xkh{\Psi\xkh{\frac{\abs{h}}{t}} \chi_{\dakh{\abs{h} > ct}}} \\
		& \leq \frac{1}{2} + \EE_{\tau}\xkh{\Psi\xkh{\frac{\abs{h}}{t}} \chi_{\dakh{\abs{h} > ct}}}, \\
		\frac{1}{2} & \leq \EE_{\tau}\xkh{\Psi\xkh{\frac{\abs{h}}{t}} \chi_{\dakh{\abs{h} > ct}}}.
	\end{align*}
	Taking $\EE_{\ell}$ on the events $\dakh{\tau = n}$ gives
	\begin{align*}
		\EE_{\ell}\xkh{\chi_{\dakh{M_{\Psi, \geq \ell} h > t}}} & = \sum_{n \geq \ell} \EE_{\ell}\xkh{\chi_{\dakh{\tau = n}}} \\
		& \leq 2 \sum_{n \geq \ell} \EE_{\ell}\xkh{\chi_{\dakh{\tau = n}} \EE_n\xkh{\Psi\xkh{\frac{\abs{h}}{t}} \chi_{\dakh{\abs{h} > ct}}}} \\
		& \leq 2 \EE_{\ell}\xkh{\Psi\xkh{\frac{\abs{h}}{t}} \chi_{\dakh{\abs{h} > ct}}}.
	\end{align*}
	Multiplying by $2t$ and integrating, we obtain
	\begin{align*}
		\EE_{\ell}\xkh{\xkh{M_{\Psi, \geq \ell} h}^2} & \leq 4 \EE_{\ell}\xkh{\int_0^{\frac{\abs{h}}{c}} t \Psi\xkh{\frac{\abs{h}}{t}} \dt}, \quad s = \frac{\abs{h}}{t}, \\
		& = 4 \xkh{\int_c^{\infty} \frac{\Psi(s)}{s^3} \dd{s}} \EE_{\ell}\xkh{\abs{h}^2} \lesssim_{\Psi} \EE_{\ell}\xkh{\abs{h}^2}. \qedhere
	\end{align*}
\end{proof}

We next prove the Carleson embedding.

\begin{prop}
	\label{prop:orlicz-embedding}
	Let $\Phi$ be a finite Young function with $\Phi(1) = 1$ and
	\[
		\int_1^{\infty} \frac{1}{\Phi(t)} \dt < \infty.
	\]
	For every Carleson sequence $a$, weight $u$, $f \in L^2(u)$ and $\ell \in \NN$, we have
	\[
		\EE_{\ell}\xkh{\sum_{n \geq \ell} a_n \frac{\abs{\EE_n\xkh{f u}}^2}{u_n^\Phi}} \lesssim_{\Phi} \norm{a}_{\car} \EE_{\ell}\xkh{\abs{f}^2 u}.
	\]
\end{prop}

\begin{proof}
	Let $\Psi$ be the complementary Young function of $t \mapsto \Phi(t^2)$. For every positive $\cF_n$-measurable $\lambda$ satisfying $\EE_n\xkh{\Psi\xkh{\frac{\abs{f} \sqrt{u}}{\lambda}}} \leq 1$, Young's inequality $ab \leq \Psi(a) + \Phi(b^2)$ gives
	\begin{align*}
		\frac{\abs{\EE_n\xkh{f u}}}{\lambda \sqrt{\norm{u}^{\cF_n}_{L^\Phi}}} & \leq \EE_n\xkh{\Psi\xkh{\frac{\abs{f} \sqrt{u}}{\lambda}} + \Phi\xkh{\frac{u}{\norm{u}^{\cF_n}_{L^\Phi}}}} \leq 2.
	\end{align*}
	Thus $\frac{\abs{\EE_n\xkh{f u}}}{\sqrt{\norm{u}^{\cF_n}_{L^\Phi}}} \leq 2\lambda$. Taking the essential infimum over $\lambda$ yields
	\[
		\frac{\abs{\EE_n\xkh{f u}}}{\sqrt{\norm{u}^{\cF_n}_{L^\Phi}}} \leq 2 \essinf\dakh{\lambda > 0 : \EE_n\xkh{\Psi\xkh{\frac{\abs{f} \sqrt{u}}{\lambda}}} \leq 1} = 2 \norm{f \sqrt{u}}^{\cF_n}_{L^\Psi}.
	\]
	Squaring both sides gives
	\[
		\frac{\abs{\EE_n\xkh{f u}}^2}{u_n^\Phi} \leq \frac{\abs{\EE_n\xkh{f u}}^2}{\norm{u}^{\cF_n}_{L^\Phi}} \leq 4 \xkh{\norm{f \sqrt{u}}^{\cF_n}_{L^\Psi}}^2.
	\]

	The inverse functions satisfy
	\[
		s \leq \sqrt{\Phi^{-1}(s)} \Psi^{-1}(s) \leq 2s.
	\]
	Using the layer cake formula and the assumption on $\Phi$, we obtain, for every fixed $c > 0$,
	\begin{align*}
		\int_c^{\infty} \frac{\Psi(t)}{t^3} \dt & \lesssim_{\Phi, c} 1 + \int_1^{\infty} \frac{1}{\Psi^{-1}(s)^2} \dd{s} \\
		& \leq 1 + \int_1^{\infty} \frac{\Phi^{-1}(s)}{s^2} \dd{s} \\
		& = 2 + \int_1^{\infty} \frac{1}{\Phi(t)} \dt < \infty.
	\end{align*}
	By Lemma~\ref{lem:carPleqMAX} and Proposition~\ref{prop:orlicz-maximal},
	\begin{align*}
		\EE_{\ell}\xkh{\sum_{n \geq \ell} a_n \frac{\abs{\EE_n\xkh{f u}}^2}{u_n^\Phi}} & \leq 4 \EE_{\ell}\xkh{\sum_{n \geq \ell} a_n \xkh{\norm{f \sqrt{u}}^{\cF_n}_{L^\Psi}}^2} \\
		& \leq 4 \norm{a}_{\car} \EE_{\ell}\xkh{\sup_{n \geq \ell} \xkh{\norm{f \sqrt{u}}^{\cF_n}_{L^\Psi}}^2} \\
		& \lesssim_{\Phi} \norm{a}_{\car} \EE_{\ell}\xkh{\abs{f}^2 u}. \qedhere
	\end{align*}
\end{proof}

Taking $f = 1$ and $a = b$ in Proposition~\ref{prop:orlicz-embedding} gives
\[
	\EE_{\ell}\xkh{\sum_{n \geq \ell} b_n \frac{u_n}{\rho_n^\Phi}} \lesssim_{\Phi} u_{\ell}.
\]
Under the normalized condition $u_n v_n \leq \frac{1}{\rho_n^\Phi}$, the principal sequence argument uses the concave function and stopping times
\[
	G(r) = 1 - r^{-\frac{1}{2}}, \qquad \gamma_{m + 1} = \inf\dakh{n > \gamma_m : b_n = 1, n < \tau, \rho_n^\Phi < \frac{1}{2} \rho_{\gamma_m}^\Phi}.
\]
Here $\gamma_0$ and $\tau$ are as in Lemma~\ref{lem:good-lambda}, and we take
\[
	t_0 = \frac{\lambda}{2}, \qquad t_m = \frac{1 - 2^{-\frac{1}{4}}}{2} \lambda(\rho_{\gamma_m}^\Phi)^{-\frac{1}{4}}, \quad m \geq 1.
\]
In this way, Theorem~\ref{thm:1side-bump} holds for finite Young functions satisfying
\[
	\int_{t_0}^{\infty} \frac{1}{\Phi_i(t)} \dt < \infty, \quad i = 1, 2,
\]
without the doubling assumption or the monotonicity assumption on $\frac{\Phi_i(t)}{t\log t}$.

\proofpart The direct Orlicz argument above also applies to other exponents. For $1 < p < \infty$, put $p' = \frac{p}{p - 1}$ and let $\Phi$ be a finite Young function satisfying
\[
	\int_{t_0}^{\infty} \xkh{\frac{t}{\Phi(t)}}^{p - 1} \frac{\dt}{t} < \infty.
\]
Under the normalized condition
\[
	u_n^{p - 1} v_n \leq (\rho_n^\Phi)^{-(p - 1)},
\]
the weak estimate and the good-$\lambda$ estimate use $p$-th powers. Keeping $\gamma_m$ as above, take
\[
	G(r) = 1 - r^{-\frac{p - 1}{2}}, \qquad t_0 = \frac{\lambda}{2}, \qquad t_m = \frac{1 - 2^{-\frac{1}{2p'}}}{2} \lambda(\rho_{\gamma_m}^\Phi)^{-\frac{1}{2p'}}, \quad m \geq 1.
\]
Integration with respect to $p\frac{\dd{\lambda}}{\lambda}$ gives
\[
	\EE_{\ell}\xkh{v \xkh{\sum_{n \geq \ell} b_n u_n}^p} \lesssim_{p, \Phi} u_{\ell}.
\]
The domination of $T_a$ by operators $T_b$ is unchanged. Applying the direct Orlicz argument to the adjoint with exponent $p'$ gives the other testing estimate, and the $L^p$ Sawyer-type testing lemma gives the corresponding $L^p(u)$ to $L^p(v)$ bound. Thus the principal sequence method also applies directly to Orlicz bump conditions at general exponents.

\subsection*{AI disclosure}

During the preparation of this work, the author used GPT, Claude, and DeepSeek to explore proof strategies, develop and refine mathematical arguments, and assist with writing and revision. The author reviewed and revised the resulting arguments and text and takes full responsibility for the mathematical content and the final manuscript.

\nocite{*}
\bibliographystyle{plain}
\bibliography{references}

\end{document}